\documentclass[11pt]{article}
\usepackage{amsmath,amssymb,amsthm,verbatim}
\usepackage{color,graphics,srcltx}
\usepackage{hyperref}
\newtheorem{theorem}{Theorem}[section]

\newtheorem{proposition}[theorem]{Proposition}
\newtheorem{corollary}[theorem]{Corollary}
\newtheorem{lemma}[theorem]{Lemma}

\newtheorem{remark}[theorem]{Remark}
\newtheorem{conjecture}[theorem]{Conjecture}
\newtheorem{problem}[theorem]{Problem}
\numberwithin{equation}{section}

\def\rddots{\cdot^{\displaystyle\cdot^{\displaystyle\cdot}}}

\def\IZ{{\mathbb Z}}
\def\IC{{\mathbb C}}
\def\IR{{\mathbb R}}

\def\cE{{\mathcal E}}

\def\la{\langle}
\def\ra{\rangle}
\def\conv{{\rm conv}}

\def\diag{{\rm diag}}
\def\re{{\rm Re\,}}
\def\im{{\rm Im\,}}
\def\rank{{\rm rank\,}}

\def\cl{{\bf cl}}
\begin{document}
\openup .88\jot
\title{Numerical radius of certain two-by-two block matrices}
\author{Hwa-Long Gau, Jia-Huo Hong, Chi-Kwong Li, Kuo-Zhong Wang}
\date{}
\maketitle

\centerline{\bf Dedicated to Professor Pei Yuan Wu.}

\begin{abstract}
We investigate the numerical range $W(T)$ and numerical radius $w(T)$ of operators of the form
$T = \begin{pmatrix} A & B \\ 0 & 0 \end{pmatrix}$. We show that $W(T)$ is the union of the
numerical ranges of a family of $2\times 2$ matrices, $T_x$, leading to several consequences,
including improved inequalities for $w(T)$. For cases where $A$ is a self-adjoint involution,
we characterize the conditions under which $W(T)$ is an elliptical disk and determine the minimum
numerical radius of $T_U = \begin{pmatrix} U^*AU & B \\ 0 & 0 \end{pmatrix}$ over all unitary operators
$U$. Finally, we study matrices $T \in M_n$ satisfying $\|T^m x\| = \|T^m\| = \|T\|$ for a unit vector
$x$ and all positive integers $m$. This analysis connects these matrices to the aforementioned block form and
provides a counterexample to the conjecture that if $\|T^k\| = \|T\|$ for all $k \ge 1$, then some power
of the matrix has a direct summand that is a scalar multiple of an idempotent.
\end{abstract}



\section{Introduction}

Let $B(H)$ be the algebra of bounded linear operators
acting on the Hilbert space $H$ equipped with the inner product $\la \cdot, \cdot \ra$.
If $H$ has dimension $n$, we identify $B(H)$ as $M_n$,
the algebra of all $n \times n$ complex matrices, and
$H = \IC^n$ equipped with the inner product $\la x, y\ra = y^*x$,
where $y^*$ denotes the conjugate transpose of $y \in \IC^n$.
Let $T \in M_n$. Its operator norm is defined by
\[
\|T\| = \sup \{ \|Tx\| : x \in H,\ \|x\| = 1 \},
\]
where $\|x\|^2 = \langle x, x \rangle$. The numerical range and numerical radius of $T$ are defined
by

\medskip
\centerline{
$W(T) = \{ \langle Tx,x \rangle: x \in H, \|x\| = 1\}$ \quad
and  \quad
$w(T) = \sup \{ |\mu|: \mu \in W(T)\}.$}

\medskip\noindent
The numerical range and numerical radius are useful tools for  studying
operators and matrices; e.g., see \cite{HJ,WG}.
In this paper, we study the numerical range, numerical radius, and some related results
for operator of the form
\begin{equation}\label{T-form}
T = \begin{pmatrix} A & B \\ 0 & 0 \end{pmatrix}.
\end{equation}
Operators of the form \eqref{T-form} appear frequently in both pure and applied areas.
For instance, if $T\in B(H)$ is idempotent, i.e., $T^2 = T$, then $T$ has an operator matrix of
the form (\ref{T-form}) with respect to an orthonormal basis using an orthonormal basis
of the range space of $T$ and its orthogonal complement.  More generally, if $\tilde{T}\in B(H)$
is a quadratic operators, i.e., $\tilde{T}^2 + b\tilde{T} + cI = 0$ for some $b, c \in \IC$,
then $\tilde{T} - \gamma I$ is unitarily similar to an operator $T$ of the form \eqref{T-form}.
By the result in \cite{TW}, $W(\tilde T) - \gamma = W(T)$
is an open or closed elliptical disk with foci at $0$ and $(b^2-4c)^{1/2}$, and with minor axis
of length $\|B\|$. Moreover, $W(T)$ is closed if and only if there
exists a unit vector $x$ such that $\|Bx\| = \|B\|$. There has also been interest in studying
the bounds for $w(T)$ for $T$ in the form \eqref{T-form}; e.g., see \cite{S} and it references.
In section 2, we will show that if $T$ has the form
\eqref{T-form}, then $W(T)$ can be written as the convex hull of elliptical disks; we also
obtain upper bound of $w(T)$ improving some results in the literature. In Section 3, we focus
on $T$ in the form \eqref{T-form} with $A, B \in M_n$ such that $A$ is a Hermitian involution,
i.e., $A = A^*$ and $A^2 = I$.

Another related topic is the
study of matrices $T \in M_n$ satisfying $\|T^m x\| = \|T^m\|=\|T\|$ for some unit vector $x$
and all positive integers $m$ that connects to
matrices of the form \eqref{T-form} with $A \in M_k$ unitary.
In this case, we may always assume that $B \in M_k$. In a sense, it is not surprising as
$T^m = \begin{pmatrix} A^m & A^{m-1}B\cr 0 & 0 \cr\end{pmatrix}$ for $m = 2, 3, \dots$.
As we shall see in Section 4, the study leads to the class of matrices
$\tilde T$ satisfying $\tilde T^*\tilde T = (\tilde T^*)^2 \tilde T^2$.
Evidently, if $\tilde T$ is invertible, then $\tilde T$ satisfies $\tilde T^* \tilde T = I$
so that $\tilde T$ is unitary. If $\tilde T$ is singular, then $\tilde T$ is unitarily similar
to a matrix of form \eqref{T-form}.
In Section 4, we will disprove the following conjecture in \cite{GWW}; see Proposition~4.10.

\begin{conjecture}\label{c11}
If $A \in M_n$ satisfies $\|A^k\| = \|A\|$ for all $k \ge 1$, then some power of $A$ has a direct summand that is a scalar multiple of an idempotent with unimodular scalar.
\end{conjecture}


\medskip

We conclude this section with some notations and basic facts about $T$ in the form
\eqref{T-form}.

For $A \in M_n$, let $A^t$ and $A^*$ denote the transpose and conjugate transpose of $A$, respectively.
The singular values of $A$ are the nonnegative square roots of the eigenvalues of $A^*A$. It is known
that $\|A\|$ is the largest singular value of $A$. If $\mathcal{F}\subseteq\{1, \dots, n\}$, we use
$A(\mathcal{F})$ to denote the principal submatrix obtained by deleting the $j$th row and $j$th column
of $A$ for all $j\not\in \mathcal{F}$. We write $\diag(a_1,\ldots,a_n)$ for the diagonal matrix with
diagonal entries $a_1,\ldots,a_n$, and denote by $0_n$ and $I_n$ the zero and identity matrices in
$M_n$, respectively.  We use $0_{k\times m}$ to denote the $k\times m$
zero matrix. For an operator $T\in B(H)$, $T$ is {\em unitarily irreducible} if it is not
unitarily similar to an operator of the form $B\oplus C$; otherwise, it is {\em unitarily reducible}.
We use $\re T$ and $\im T$ to denote the real part $(T+T^*)/2$ and imaginary part $(T-T^*)/(2i)$ of $T$,
respectively. $T$ is {\em positive semidefinite}, denoted by $T\ge 0$, if $\langle Tx,x\rangle\ge 0$ for all $x\in H$. If
$T_1$, $T_2\in B(H)$ are self-adjoint, then $T_1\ge T_2$ means that $T_1-T_2\ge 0$. For any subset $\bigtriangleup$ of complex plane $\IC$, $\cl(\bigtriangleup)$ denotes the closure of $\bigtriangleup$.

In our discussion, it is convenient to embed $H$ into $H\oplus H$ and assume that
$A, B \in B(H)$. We may further replace $T$ by $U^*TU = \begin{pmatrix} U_1^*AU_1 & U_1^*BU_2\cr
0 & 0 \cr\end{pmatrix}$ such that $U_1^*BU_2 \in B(H)$ is positive semidefinite.
In the finite dimensional case, we may assume that $U_1^*BU_2 = \diag(b_1, \dots, b_n)$
with $b_1 \ge \cdots \ge b_n \ge 0$.

\section{Generating ellipses and consequences}

It is known that for any operator $T \in B(H)$, $W(T)$ can be written as the union
of $W(X^*TX)$ for $X: \IC^2 \rightarrow H$ with $X^*X = I_2$.
We will show that if $T$ has the form \eqref{T-form},
then $T$ is the union of $W(X^*TX)$ some special $X$ so that $X^*TX$ has a simple form.
We begin with the following easy case.

\begin{proposition}  Let $T$ have the form \eqref{T-form}. Suppose $A, B \in M_k$ and there are unitary $U, V$ such that
$U^*AU = \diag(a_1, \dots, a_k)$ and $U^*BV = \diag(b_1, \dots, b_k)$.
Let $T_j = \begin{pmatrix} a_j & b_j \cr 0 & 0 \cr\end{pmatrix}$ for $j = 1, \dots, k$.
Then
$$W(T) = \conv \Big(\bigcup_{j=1}^k W(T_j)\Big) \quad \mbox{and} \quad
w(T) = \max\{ w(T_j): 1 \le  j \le k\}.$$
\end{proposition}

\it Proof. \rm This follows from the fact that the assumption ensures that $T$
is a unitarily similar to a direct sum of $T_j$ so that $W(T)=\conv \left(\cup_{j=1}^k W(T_j)\right)$
and $w(T) = \max \{ w(T_j): 1 \le  j \le k \}.$ \qed

\smallskip
In general, we can write $W(T)$ as the union of $W(T_x)$ for some upper triangular
matrix $T_x \in M_2$ even if $T$ does not have the nice special form in the last proposition.

\begin{theorem}\label{p22}
Let $T = \begin{pmatrix} A & B \cr 0 & 0 \cr\end{pmatrix} \in B(H \oplus H)$.
Then
$$W(T) =\conv  \bigcup \{W(T_x): x \in H, \|x\| = 1\} \quad \hbox{ where } T_x
=  \begin{pmatrix}\la Ax,x\ra & \|B^*x\| \cr 0 & 0 \cr\end{pmatrix}.$$
Consequently,
\begin{eqnarray*}
w(T) &=& \sup \{ w(T_x): x \in B(H), \|x\| = 1\} \\
&=& \sup \{ \frac{1}{2}(|\la Ax, x\ra| + \sqrt{|\la Ax,x\ra|^2 + \|B^*x\|^2}): x \in H, \|x\| = 1\}  \\
&\le& \frac{1}{2}w(A) + \frac{1}{2} \min\{ \sqrt{w(AA^*+ BB^*)}, \sqrt{ w(A)^2 + \|B\|^2}\}.
\end{eqnarray*}
The equality holds if and only if there is a sequence of unit vectors $\{x_m\}$
such that $|\langle Ax_m, x_m\rangle| \rightarrow w(A)$ and

{\rm (1)}  $\|B^*x_m\|^2 \rightarrow  w(AA^*+BB^*) - w(A)^2$, or

{\rm (2)} $\|B^*x_m\| \rightarrow  \|B\|$ and $ w(AA^*+BB^*) \ge w(A)^2 + \|B\|^2$.
\end{theorem}

\it Proof. \rm Every unit vector $z \in H\oplus H$ can be written as $\cos \theta x \oplus \sin \theta y$
for some unit vectors $x, y \in H$. Thus, $\la Tz,z\ra = (\cos\theta \ \sin \theta)
\tilde T_x (\cos\theta \ \sin \theta)^t$, where
$\tilde T_x = \begin{pmatrix} \la Ax,x\ra & \la By,  x\ra\cr 0 & 0 \cr\end{pmatrix}$.
Note that $W(\tilde T_x)$ is an elliptical disk with foci $\la Ax,x\ra$ and $0$
with minor axis having length $|\la By,x\ra|$.
We may choose $y = B^*x/\|B^*x\|$ to get $T_x$ as described in the theorem
so that $W(\tilde T_x) \subseteq W(T_x)$. Thus, the set
equality holds.

Now, for any  unit vector $x \in H$,
the numerical radius of $w(T_x)$ is attained at the end point of the major axis and
\begin{eqnarray}\label{eq21}
w(T_x)
 &=&\frac{1}{2}\Big(|\la Ax,x\ra| + \sqrt{|\la Ax,x\ra|^2 + \|B^*x\|^2}\Big) \nonumber\\
&\le&  \frac{1}{2}\Big(|\la Ax,x\ra| + \sqrt{\|A^*x\|^2 + \la BB^*x,x\ra}\Big) \\
&=&  \frac{1}{2}\Big(|\la Ax,x\ra| + \sqrt{\la AA^*x, x\ra + \la BB^*x, x\ra}\Big)  \nonumber\\
&\le&   \frac{1}{2}\Big(w(A) + \sqrt{w(AA^*+BB^*)}\Big).\nonumber
\end{eqnarray}
Clearly, we also have
\begin{equation}\label{eq22}
w(T_x) =\frac{1}{2}\Big(|\la Ax,x\ra| + \sqrt{|\la Ax,x\ra|^2 + \|B^*x\|^2}\Big)
\le \frac{1}{2} \Big(w(A) + \sqrt{w(A)^2+\|B\|^2}\Big).\end{equation}
The equality case could be verified readily by \eqref{eq21} and \eqref{eq22}.
\qed

\medskip
Note that the two quantities $\sqrt{w(AA^*+BB^*)}$ and $\sqrt{w(A)^2 + \|B\|^2}$ are not comparable
in general.
For example,
\begin{itemize}
\item[(a)] if $A = \left({\textstyle 0\atop  \textstyle 0} \ {\textstyle 2 \atop \textstyle 0}\right)$ and $B = \left({\textstyle 1\atop  \textstyle 0} \ {\textstyle 0 \atop \textstyle 0}\right)$, then
$w(AA^* + BB^*) = 5 \ge 2 = w(A)^2 + \|B\|^2$;
\item[(b)] if $A = \left({\textstyle 1\atop  \textstyle 0} \ {\textstyle 0 \atop \textstyle 0}\right)$ and $B=\left({\textstyle 0\atop  \textstyle 0} \ {\textstyle 0 \atop \textstyle 1}\right)$,
then $w(AA^*+BB^*) = 1 \le 2 = w(A)^2 + \|B\|^2$.
 \end{itemize}

\medskip
We have the following corollary, which is  Theorem 3.1 in \cite{S}.
\begin{corollary}  Suppose $T = \begin{pmatrix} A & B \cr 0 & 0 \cr\end{pmatrix} \in B(H \oplus H)$.
Then $w(T) \le \frac{1}{2} (\|A\| + \sqrt{\|AA^* + BB^*\|})$.
\end{corollary}

We also have the following.

\begin{corollary} \label{2.4} Let
$T = \begin{pmatrix} A & B \cr 0 & 0 \cr\end{pmatrix} \in B(H \oplus H)$.
If $B_1, B_2\in B(H)$ such that $B_1B_1^* \le BB^* \le B_2B_2^*$ and $T_j$ is obtained
from $T$ be replacing $B$ with $B_j$ for $j = 1,2$, then
$$W(T_1) \subseteq W(T) \subseteq W(T_2).$$
\end{corollary}

\it Proof. \rm
Note that for any unit vector $x$, $W(T_x)$ is an elliptical disk with foci
$0, \la Ax,x\ra$ and minor axis of length $\|B^*x\|$. For $T_j$ the corresponding
$W((T_j)_x)$ is an elliptical disk with the same foci as $W(T_x)$ and minor axis
of length $\|B_j^*x\|$. Since $B_1B_1^* \le BB^* \le B_2 B_2^*$, we see that
$\|B_1^*x\| \le \|B^*x\| \le \|B_2^*x\|$. Hence,
$W((T_1)_x) \subseteq W(T_x) \subseteq W((T_2)_x)$.
Taking the union of the sets for all unit vectors $x$, we get the inclusion. \qed

\begin{corollary} Let $T$ have the form \eqref{T-form} be such that $B = bI$ with $b \ge 0$ and $W(A)$ is closed.
Then
$$W(T)  = \bigcup \{W(T_x): x\in H,\, \|x\|=1,\,\langle Ax,x\rangle \hbox{ is a boundary point of } W(A)\}.$$
\end{corollary}

\it Proof. \rm Clearly, $W(T_x)$ has the same minor axis for every unit vector $x$.
Hence, the right supporting line of $W(T)$ is determined by the endpoint of the major axis of $W(T_x)$.
Therefore, we only need to focus on those $T_x$ for which $\langle Ax,x\rangle$ is a boundary point of $W(A)$.
\qed

If $A\in M_2$ is unitary, after a rotation,
then $A$ is unitarily similar to $\cos\theta I_2 + i\sin\theta(E_{11}-E_{22})$, where $E_{11}=\left({\textstyle 1\atop  \textstyle 0} \ {\textstyle 0 \atop \textstyle 0}\right)$ and $E_{22}=\left({\textstyle 0\atop  \textstyle 0} \ {\textstyle 0 \atop \textstyle 1}\right)$.
We have the following.

\begin{corollary}
Let $T$ have the form \eqref{T-form} with $A, B \in M_2$ such that
$A = \cos\theta I_2 + i\sin \theta (E_{11}-E_{22})$ with $\theta \in [0,\pi/2]$.
Then $W(T) = \bigcup \{W(T_r): r \in [-\sin\theta , \sin\theta]\}$, where $T_r = \begin{pmatrix} \cos\theta + ir & c \cr 0 & 0 \cr
\end{pmatrix}$, and $(r, c^2)$ ranges over the joint numerical range
$$
W(\im A,\, BB^*) = \{ (\langle (\im A)x,x\rangle,\, \langle BB^* x,x\rangle ) : \|x\|=1 \}.
$$
\end{corollary}

Next, we consider maximum $\sigma(\re (e^{i\theta}T))$, which allows us
to determine  the support lines of $\cl(W(T))$, the closure of $W(T)$, and reduces
the computation of $w(T)$ to an optimization problems over a parameter $\theta \in [0, 2\pi)$.

\begin{theorem} \label{support-line}
Suppose $T\in B(H \oplus H)$ has the form {\rm (\ref{T-form})} with $A, B \in B(H)$.
Let
$$T_\theta = \re(e^{i\theta}T) = \frac{1}{2}
\begin{pmatrix} e^{i\theta} A + e^{-i\theta} A & e^{i\theta} B \cr
e^{-i\theta} B^* & 0 \cr\end{pmatrix}, \qquad \theta \in [0, 2\pi).$$
Then
\begin{eqnarray}\label{eq23}
 \lambda_{\max}(T_\theta) &=&
\inf \{ \lambda \ge 0: 4\lambda\big(\lambda I - \re (e^{i\theta} A)\big)
\ge BB^*\} \nonumber\\
& \ge &   \frac{1}{2} \lambda_{\max}\Big( \re(e^{i\theta}A) +
\sqrt{(\re(e^{i\theta}A))^2 + BB^*}\Big).
\end{eqnarray}
and
\begin{eqnarray*}w(T)&=&\max_{0\le \theta<2\pi}\inf \{ \lambda \ge 0: 4\lambda\big(\lambda I - \re (e^{i\theta} A)\big)
\ge BB^*\}\\
&\ge & \frac{1}{2} \max_{0\le \theta< 2\pi} \lambda_{\max}\Big( \re(e^{i\theta}A) +
\sqrt{(\re(e^{i\theta}A))^2 + BB^*}\Big).\end{eqnarray*}
\end{theorem}

\it Proof. \rm Note that $T_{\theta}$ is
unitarily similar to
$\frac{1}{2}
\begin{pmatrix} e^{i\theta} A + e^{-i\theta} A & B \cr
B^* & 0 \cr\end{pmatrix}$ and $0\in W(T_{\theta})$. Therefore,
$$\lambda_{\max}(T_\theta) = \inf\{\lambda\ge 0 : \lambda I - T_\theta \ge 0\}$$
and
$$\lambda I - T_\theta =  \begin{pmatrix} \lambda I - \re(e^{i\theta} A) & -B/2\cr
-B^*/2 & \lambda I\cr\end{pmatrix},$$
which is positive semidefinite if and only if
$\lambda I - \re(e^{i\theta} A) \ge BB^*/(4\lambda)$. The first equality holds.

Note that $\lambda^2 I - \lambda \re(e^{i\theta} A) -  BB^*/4\ge 0$ is equivalent to
$$\left(\lambda I - \frac{1}{2}\re(e^{i\theta}A)\right)^2
\ge  \frac{1}{4} \left( (\re(e^{i\theta}A))^2 + BB^* \right).$$
Since $\lambda I - T_\theta\ge 0$,
we see that $\lambda I - \re(e^{i\theta}A)\ge 0$ and $\lambda \ge 0$.
So,  $\lambda I - \frac{1}{2}\re(e^{i\theta}A) \ge 0$.
It is  known that if $P \ge  Q \ge 0$ then $\sqrt P \ge \sqrt Q$.
So, $\lambda I - \frac{1}{2}\re(e^{i\theta}A)
\ge \frac{1}{2} \sqrt {(\re(e^{i\theta}  A))^2+BB^*}.$
This shows that
$$\Big\{ \lambda \ge 0:  \lambda I - T_\theta \ge 0\Big\}\subseteq \Big\{ \lambda \ge 0: \lambda I - \frac{1}{2}\Big(\re(e^{i\theta}A)
+ \sqrt {(\re(e^{i\theta}  A))^2+BB^*}\Big)\ge 0\Big\}.$$
That is,
$$\inf \Big\{ \lambda \ge 0:  \lambda I - T_\theta \ge 0\Big\}\ge \inf \Big\{ \lambda \ge 0: \lambda I - \frac{1}{2}\Big(\re(e^{i\theta}A)
+ \sqrt {(\re(e^{i\theta}  A))^2+BB^*}\Big)\ge 0\Big\}.$$
The inequality \eqref{eq23} holds. The final inequality follows from $w(T)=\max_{0\le \theta< 2\pi} \lambda_{\max}(T_{\theta})$.
\qed

\medskip
Note that Theorem \ref{p22} provides upper bounds for $w(T)$ whereas Theorem \ref{support-line} provides a lower bound for $w(T)$. Moreover, in Theorem \ref{p22}, the computation of $w(T)$ reduces to a maximization problem
of $w(T_x)$ for
$x \in H$ with $\|x\| = 1$. In Theorem \ref{support-line}, the computation of $w(T)$ reduced to a
maximization problem
$$\max_{0\le \theta< 2\pi} \inf \{ \lambda \ge 0: 4\lambda\big(\lambda I - \re (e^{i\theta} A)\big)
\ge BB^*\}$$
with two parameters $\lambda \ge 0$ and $\theta \in [0, 2\pi)$. Furthermore,
for a fixed $\theta \in [0, 2\pi)$, one may consider $\lambda > 0$ in a compact set to determine
$$\lambda_\theta = \inf\{\lambda \ge 0:
4\lambda\big(\lambda I - \re (e^{i\theta} A)\big)
\ge BB^*\}$$
and then solve the original optimization by taking $\max\{ \lambda_\theta: \theta \in [0,2\pi]\}$.

\medskip

If $A = \diag(a_1, \dots, a_k)$ and  $B = \diag(b_1, \dots, b_k) \in M_k$, then
\begin{eqnarray*}
\lambda_\theta
&=&  \inf\{ \lambda\ge 0: 4\lambda\big(\lambda - \re(e^{i\theta} a_j) \big)\ge |b_j|^2, j= 1, \dots, k\}\cr
&=&
\frac{1}{2} \max\{ \re(e^{i\theta}a_j) + \sqrt{ (\re(e^{i\theta}a_j))^2 + |b_j|^2}: j = 1, \dots, k\}.
\end{eqnarray*}

A similar formula remains valid for infinite-dimensional operators whenever $\re A$, $\im A$, and $BB^*$ commute, by the spectral theorem.

\section{When the $(1,1)$ block is a self-adjoint involution}

In this section, we consider the special case when
$A$ is a self-adjoint involution operator,
i.e., $A = A^*$ and $A^2= I$.
In such a case, we may assume apply a unitary similar of the form $U = U_1 \oplus U_2$ to $T$,
and assume $A = I_{H_1} \oplus -I_{H_2}$ so that $(\re(e^{i\theta}  A))^2 = \cos^2\theta I$.
By Theorem \ref{support-line}
and the discussion after it,
$$w(T) = \max_{\theta \in [0, 2\pi]}
\inf \{ \lambda \ge 0: 4\lambda\big(\lambda I - \re (e^{i\theta} A)\big)  \ge BB^* \}.$$
In particular, if $A\in M_k$ and $B = \diag(b_1, \dots, b_k)$, then
$$\max\sigma(T_\theta)=
\frac{1}{2}\max\sigma\left(\cos \theta A + \diag\Big(\sqrt{\cos^2\theta +| b_1|^2},
\dots, \sqrt{\cos^2\theta + |b_k|^2}\Big)\right).$$

Suppose  $T$ has the form \eqref{T-form}
such that $A=A^*$ and $A^2 = I$.
We determine the condition on $T$ such that $W(T)$ is an elliptical disk in the following.

\begin{theorem} \label{3.1}
Let  $T=\left(
         \begin{array}{cc}
           A & B \\
           0 & 0 \\
         \end{array}
       \right) \in M_n
$ such that $A = A^* \in M_k$ with $1\le k<n$
satisfying  $A^2 = I_k$ and $B\neq 0$.
The following statements are equivalent.
\begin{itemize}
\item[{\rm (a)}]  $W(T)$ is an elliptical disk.
\item[{\rm (b)}] The matrix $T$ is unitarily similar to a direct sum $T_1\oplus T_2$ with
$W(T_2) \subseteq W(T_1)$, where $T_1$ is either
$$
\left(
  \begin{array}{cc}
    a & \|B\| \\
    0 & 0 \\
  \end{array}
\right)\ \mbox{with } a\in\{1, -1\} \quad \mbox{ or }\quad \left(
                    \begin{array}{ccc}
                      0 & 1 & \|B\| \\
                      1 & 0 & 0 \\
                      0 & 0 & 0 \\
                    \end{array}
                  \right).
$$
\end{itemize}
\end{theorem}

To prove Theorem \ref{3.1}, we need the following lemma.

\begin{lemma}\label{p31}
Let $\theta\in [-\pi,\pi]$, $\alpha_1>\alpha_2\ge 0$,
$$A_\theta=\left(
                                            \begin{array}{cc}
                                              \cos\theta & \sin\theta \\
                                              \sin\theta & -\cos\theta \\
                                            \end{array}
                                          \right),  \ \ \
                                          D=\begin{pmatrix} \alpha_1 & 0 \\ 0 & \alpha_2\end{pmatrix} \ \ \mbox{ and } \ \ T_\theta=\left(
                \begin{array}{cc}
                  A_\theta & D \\
                  0_2 & 0_2 \\
                \end{array}
              \right).$$
Then
$$ w(T_{\pm\pi/2})
=\frac{\sqrt{4+\alpha_1^2+\alpha_2^2+\sqrt{16+8\alpha_1^2+8\alpha_2^2+(\alpha_1^2-\alpha_2^2)^2}}}{2\sqrt{2}} < w(T_\theta) \  \mbox{ for all } \theta\neq \pm \frac{\pi}{2}.
$$
\end{lemma}
\begin{proof}
Observe that $T_\theta$ is unitarily similar to $T_{-\theta}$, so it suffices to consider $\theta\in [0,\pi]$.
For any $\mu\in [0,2\pi]$, $e^{i\mu}T_\theta$ is unitarily similar to
$\left(\begin{array}{cc}
e^{i\mu}A_\theta & D \\
0 & 0 \\
\end{array} \right)$.
Therefore,
\begin{eqnarray*}
&&\mbox{det}(\re(e^{i\mu}T_\theta)-\lambda I_4)=\mbox{det}\left(
                                                                  \begin{array}{cccc}
                                                                    \cos\theta\cos\mu-\lambda & \sin\theta\cos\mu & \frac{\alpha_1}{2} & 0 \\
                                                                    \sin\theta\cos\mu & -\cos\theta\cos\mu-\lambda & 0 & \frac{\alpha_2}{2} \\
                                                                    \frac{\alpha_1}{2} & 0 & -\lambda & 0 \\
                                                                    0 & \frac{\alpha_2}{2} & 0 & -\lambda \\
                                                                  \end{array}
                                                                \right)\\
&=&-\lambda \mbox{det}\left(
                        \begin{array}{ccc}
                          \cos\theta\cos\mu-\lambda& \sin\theta\cos\mu & \frac{\alpha_1}{2} \\
                           \sin\theta\cos\mu &  -\cos\theta\cos\mu-\lambda & 0 \\
                          \frac{\alpha_1}{2} & 0 & -\lambda  \\
                        \end{array}
                      \right)-\frac{\alpha_2^2}{4}\mbox{det}\left(
                                                              \begin{array}{cc}
                                                                \cos\theta\cos\mu-\lambda & \frac{\alpha_1}{2} \\
                                                                \frac{\alpha_1}{2} &  -\lambda \\
                                                              \end{array}
                                                            \right)\\
&=&\lambda^2\mbox{det}\left(
                        \begin{array}{cc}
                         \cos\theta\cos\mu-\lambda& \sin\theta\cos\mu \\
                       \sin\theta\cos\mu &  -\cos\theta\cos\mu-\lambda \\
                        \end{array}
                      \right)-\lambda\frac{\alpha_1^2}{4}(\cos\theta\cos\mu+\lambda)-\frac{\alpha_2^2}{4}(\lambda^2-\lambda \cos\theta\cos\mu-\frac{\alpha_1^2}{4})\\
&=&\lambda^4-\lambda^2(\cos^2\mu+\frac{\alpha_1^2+\alpha_2^2}{4})+\lambda\frac{\alpha_2^2-\alpha_1^2}{4}\cos\theta\cos\mu+\frac{\alpha_1^2\alpha_2^2}{16}\\
&\equiv&F_{\theta,\mu}(\lambda).
\end{eqnarray*}
Let $\lambda_1(F_{\theta,\mu})$ denote the largest real root of $F_{\theta,\mu}$. Then $w(T_\theta)=\max_{\mu\in [0,2\pi]}\lambda_1(F_{\theta,\mu})$.
For $\theta=\pi/2$, since $T_{\pi/2}$ is a nonnegative matrix, we get
\begin{eqnarray*}
\beta & \equiv & w(T_{\pi/2})=\lambda_1(F_{\pi/2,0})=\max\{\lambda :\lambda^4-(1+\frac{\alpha_1^2+\alpha_2^2}{4})\lambda^2+\frac{\alpha_1^2\alpha_2^2}{16}=0\}\\
& = & \frac{\sqrt{4+\alpha_1^2+\alpha_2^2+\sqrt{16+8\alpha_1^2+8\alpha_2^2+(\alpha_1^2-\alpha_2^2)^2}}}{2\sqrt{2}}.
\end{eqnarray*}
For $\theta\in [0,\pi]\setminus \{\pi/2\}$,
$$
F_{\theta,\mu}(\lambda)=F_{\pi/2,0}(\lambda)+\lambda^2\sin^2\mu+\lambda\frac{\alpha_2^2-\alpha_1^2}{4}\cos\theta\cos\mu.$$
Evaluating this at $\lambda=\beta$, we find
$$
F_{\theta,\mu}(\beta)=\beta^2\sin^2\mu-\frac{\alpha_1^2-\alpha_2^2}{4}\beta\cos\theta\cos\mu.
$$
Since $\alpha_1>\alpha_2\ge 0$, we have $F_{\theta, 0}(\beta)<0$ for all $\theta\in[0, \pi/2)$, and, $F_{\theta, \pi}(\beta)<0$ for all $\theta\in(\pi/2, \pi]$, that is, $ \lambda_1(F_{\theta, 0})>\beta$ for all $\theta\in[0, \pi/2)$, and, $ \lambda_1(F_{\theta, \pi})>\beta$ for all $\theta\in(\pi/2, \pi]$. Therefore,
$$w(T_\theta)=\max_{\mu\in [0,2\pi]}\lambda_1(F_{\theta,\mu})\ge \max\{\lambda_1(F_{\theta, 0}), \lambda_1(F_{\theta, \pi})\}>\beta$$
for all  $\theta\in [0,\pi]\setminus \{\pi/2\}$. This completes the proof.
\end{proof}

We are now ready to prove Theorem \ref{3.1}.

\begin{proof}[Proof of Theorem \ref{3.1}]
If $A=\pm I_k$, then $T$ is a quadratic operator and the assertion follows from \cite[Theorem 2.1]{TW}. Therefore, we may assume that $A$ is unitarily similar to $I_{k_1}\oplus (-I_{k_2})$, where $1\le k_1\le k-1$ and $k_2=k-k_1$. Let $m=n-k$, then $B$ is a $k\times m$ matrix. Say, $B=V_0S_B V_1^*$ is the singular value decomposition of $B$, where $V_0\in M_k$ and $V_1\in M_m$ are unitary, and $S_B=[\|B\|]\oplus B_1$ for some $B_1\in M_{k-1, m-1}$. Let $V=V_0\oplus V_1\in M_n$ be unitary, then $V^*TV=\begin{pmatrix} V_0^*AV_0 & S_B\\ 0 & 0\end{pmatrix}$, hence we may assume that $B=[\|B\|]\oplus B_1$ for some $B_1\in M_{k-1, m-1}$.

$(a)\Rightarrow (b)$. Assume that $W(T)$ is an elliptical disk. Since the geometric multiplicity and algebraic multiplicity of each eigenvalue of $T$ are equal, it follows from \cite[Corollary 4]{W} that $W(T)$ is not a circular disk. Therefore, the foci of $W(T)$ must be one of
$\{0,1\}$, $\{0,-1\}$, and $\{-1,1\}$. Moreover, since
$\im T=\left(
                          \begin{array}{cc}
                            0 & B/(2i) \\
                            -B^*/(2i) & 0 \\
                          \end{array}
                        \right)
$, it implies that $\max \sigma (\im T)=\|B\|/2=-\min \sigma (\im T)$, that is, the length of minor axis
of $W(T)$ is $\|B\|$, we infer that either $$W(T)=W(\left(
                                                                 \begin{array}{cc}
                                                                   \pm 1 & \|B\| \\
                                                                   0 & 0 \\
                                                                 \end{array}
                                                               \right)
)\ \ \ \mbox{ or } \ \ \ W(T)= W(\left(
                \begin{array}{cc}
                  1 & \|B\| \\
                  0 & -1 \\
                \end{array}
              \right)
).$$
Since $A^*=A$ and $A^2=I_k$, we may write $A=\left(
                                                     \begin{array}{cc}
                                                       \cos\theta & * \\
                                                       * & A_1 \\
                                                     \end{array}
                                                   \right)
$, where $A_1\in M_{k-1}$ is self-adjoint. By interlacing property, there exists a unitary matrix $U_1\in M_{k-1}$ such that $U_1^*A_1U_1=[\alpha]\oplus I_{k_1-1}\oplus (-I_{k_2-1})$, where $\alpha\in [-1, 1]$. Write $U_1^*B_1=[b_{ij}]\in M_{k-1, m-1}$.
Set $U_2=[1]\oplus U_1\in M_k$ and define $U=U_2\oplus I_m$, so that $U\in M_{n}$ is unitary.  Then $$
U^*TU=\left(
        \begin{array}{cc}
          U_2^*AU_2 & U_2^*B \\
          0 & 0 \\
        \end{array}
      \right)\equiv T',
$$
where
$$U_2^*AU_2=\left(\begin{array}{c|c} \begin{array}{cc} \cos\theta & a_{12}\\ a_{21} & \alpha \end{array} & 0 \\ \hline 0 & \begin{array}{cc} I_{k_1-1} & \\  & -I_{k_2-1}\end{array} \end{array}\right) \ \mbox{ and } \
U_2^*B=\left(\begin{array}{cc|cc} \|B\| & 0 &  \cdots & 0 \\ 0 & b_{11} &  \cdots & b_{1, m-1} \\ \hline \vdots & \vdots &  & \vdots\\ 0 & b_{k-1, 1} & \cdots & b_{k-1, m-1}\end{array}\right).$$
Note that the submatrix $\left({\textstyle \cos\theta \atop  \textstyle a_{21}} \ {\textstyle a_{12} \atop \textstyle \alpha}\right)$ is unitarily similar to $\diag(1, -1)$, thus $\alpha=-\cos\theta$ and $\overline{a}_{21}=a_{12}=e^{i\mu}\sin\theta$ for some real $\mu$, after a unitary similarity by a diagonal unitary matrix, we may assume that $a_{12}=a_{21}=\sin\theta$. On the other hand, the horizontal line $y=\|B\|/2$ is tangent to the elliptical disk $W(T)$ at the point $c+i\|B\|/2$, where $c\in\IR$ is the center of $W(T)$.
Define $x_1=[1\;0\ldots 0]^t\in \IC^k$, $x_2=[i\; 0\ldots 0]^t\in \IC^{m}$, and $x=[x_1\;x_2]^t/\sqrt{2}\in \IC^n$. Then $\|x\|=1$ and $\langle \im(T')x,x\rangle=\|B\|/2$,
so the real part of $\langle T'x,x\rangle$ gives the center of the elliptical disk $W(T)$, or $c=\langle \re(T')x,x\rangle$.
Compute
$$
c=\langle \re(T')x,x\rangle=\langle \left(
                                         \begin{array}{cc}
                                           \cos\theta & \frac{\|B\|}{2} \\
                                           \frac{\|B\|}{2} & 0 \\
                                         \end{array}
                                       \right)\left(
                                                \begin{array}{c}
                                                  \frac{1}{\sqrt{2}} \\
                                                  \frac{i}{\sqrt{2}} \\
                                                \end{array}
                                              \right),
\left(
                                                \begin{array}{c}
                                                  \frac{1}{\sqrt{2}} \\
                                                  \frac{i}{\sqrt{2}} \\
                                                \end{array}
                                              \right)
\rangle=\frac{\cos\theta}{2}.
$$
If $W(T)=W(\left(
             \begin{array}{cc}
               \pm 1 & \|B\| \\
               0 & 0 \\
             \end{array}
           \right)
)$, then $(\cos\theta)/2=c=\pm1/2$ and so, $\cos\theta=\pm 1$.
By interchanging the 2nd and $(k+1)$-th rows and columns of the matrix $T'$, we obtain
$\begin{pmatrix} \pm 1 & \|B\| \cr 0 & 0 \cr \end{pmatrix} \oplus T_2$ as asserted.

Now suppose that $W(T)=W(\left(
                             \begin{array}{cc}
                               1 & \|B\| \\
                               0 & -1 \\
                             \end{array}
                           \right)
)$, then $\cos\theta=2c=0$.
After a unitary similarity by a diagonal unitary matrix, we may assume that the $2\times 2$ leading principal submatrix of $U_2^*AU_2$ is $\left(                                                                                               \begin{array}{cc}                                                                                                   0 & 1 \\                                                                                                   1 & 0 \\                                                                                                 \end{array}                                                                                               \right)$ and $b_{1j}\ge 0$ for all $1\le j\le m-1$.
We now claim that $b_{1j}=0$ for all $1\le j\le m-1$.  Indeed, for each $j=1, \dots, m-1$, consider the $4\times 4$ principal submatrix
$$T'(\{1, 2, k+1, k+1+j\})=
\left(
  \begin{array}{cccc}
    0 & 1 & \|B\| & 0 \\
    1 & 0 & 0 & b_{1j} \\
    0 & 0 & 0 & 0 \\
    0 & 0 & 0 & 0 \\
  \end{array}
\right)\equiv  E_j.$$
We know that $W(E_j)\subseteq W(T')=W(\left(
                        \begin{array}{cc}
                          1 & \|B\| \\
                          0 & -1\\
                        \end{array}
                      \right)
)$, hence
$$w(E_j)\le w(\left(
                        \begin{array}{cc}
                          1 & \|B\| \\
                          0 & -1\\
                        \end{array}
                      \right))=\max \sigma (\left(
                                              \begin{array}{cc}
                                                1 & \frac{\|B\|}{2} \\
                                                \frac{\|B\|}{2} & -1 \\
                                              \end{array}
                                            \right)
                      )=\frac{\sqrt{4+\|B\|^2}}{2}.$$
On the other hand, Lemma \ref{p31} yields that
$$
w^2(E_j)=\frac{4+\|B\|^2+b_{1j}^2+\sqrt{16+8\|B\|^2+8b_{1j}^2+(\|B\|^2-b_{1j}^2)^2}}{8}.
$$
If $b_{1j}\neq 0$, then
$$
w^2(E_j)-\frac{4+\|B\|^2}{4}=\frac{1}{8}\Big(\sqrt{16+8\|B\|^2+8b_{1j}^2+(\|B\|^2-b_{1j}^2)^2}-(4+\|B\|^2-b_{1j}^2)\Big)>0,
$$
a contradiction. We conclude that $b_{1j}=0$ for all $j$.
By interchanging the 3rd and $(k+1)$-th rows and columns of $T'$, we obtain
$$
\left(
  \begin{array}{ccc}
    0 & 1 & \|B\| \\
    1 & 0 & 0 \\
    0 & 0 & 0 \\
  \end{array}
\right)\oplus T_2.
$$

$(b)\Rightarrow (a)$. Let $T_1=\left(
                    \begin{array}{ccc}
                      0 & 1 & \|B\| \\
                      1 & 0 & 0 \\
                      0 & 0 & 0 \\
                    \end{array}
                  \right)$.
We only need to show $W(T_1)$ is an elliptical disk. Since $T_1$ is unitarily similar to
$$
\left(
  \begin{array}{ccc}
    1 & 0 & \|B\|/\sqrt{2} \\
    0 & -1 & \|B\|/\sqrt{2} \\
    0 & 0 & 0 \\
  \end{array}
\right),
$$
by \cite[Theorem 2.2]{KRS}, it is easy to check that the numerical range of the above matrix is an elliptical disk.
This completes the proof.
\end{proof}

It is not hard to extend Theorem \ref{3.1} to compact operators in $B(H)$.
In fact,  one can extend the result to general operators.

\begin{theorem} \label{3.4}
 Suppose $T \in B(H\oplus H)$ has the form \eqref{T-form} such that
$A= A^*$ and $A^2 = I$.
Then  $\cl(W(T))$ is an  elliptical disk ${\mathcal E}$ with foci
in $\{0, -1,1\}$ if and only if  one of
the following holds.
\begin{itemize}
\item[{\rm (a)}] There is a unit vector $x\in H$ such that
$\|B^*x\| = \|B\|$ and $T$ is unitarily similar to $T_1 \oplus T_2$
      with $W(T_2) \subseteq W(T_1)$, where
$$
T_1 = \begin{pmatrix} a & \|B\|\cr 0 & 0\cr\end{pmatrix} \hbox{ with } a \in \{1, -1\} \quad
             \hbox{ or }   \quad
    T_1 = \begin{pmatrix} 0 & 1 & \|B\| \cr  1 & 0 &  0\cr  0 & 0 & 0\cr\end{pmatrix}.$$
\item[{\rm (b)}] There is no unit vectors $x\in H$ such that $\|B^*x\| = \|B\|$,  and there is a sequence of unit
vectors $\{x_m\}$ in $H$ such that one of the following holds.

       \smallskip
        \noindent
        {\rm (b.1)} $\begin{pmatrix} \la Ax_m,x_m\ra &  \|B^*x_m\|\cr  0 & 0 \end{pmatrix}
        \rightarrow  T_0 = \begin{pmatrix} a & \|B\| \cr 0 & 0 \cr \end{pmatrix}$, $a \in \{1, -1\}$,
        and $W(T_0) = \cl (W(T)) = \cE$.

        \smallskip
        \noindent
         {\rm (b.2)}
         $\begin{pmatrix} \la Ax_m,x_m\ra & \sqrt{1-\la Ax_m,x_m\ra^2} &  \|B^*x_m\|\cr   \sqrt{1-\la Ax_m,x_m\ra^2} & -\la Ax_m,x_m\ra & 0 \cr 0 & 0 & 0 \cr \end{pmatrix}
        \rightarrow  T_0 = \begin{pmatrix} 0 & 1 & \|B\| \cr 1 & 0 & 0 \cr 0 & 0 & 0 \cr\end{pmatrix}$
        and

\qquad        $W(T_0) = \cl (W(T)) = \cE$.
\end{itemize}
\end{theorem}

\noindent
\it Proof. \rm
If $A=\pm I$, then $T$ is a quadratic operator and the assertion follows from \cite[Theorem 2.1]{TW}. Therefore, we may assume that $A\neq\pm I$.
The sufficiency proof is clear.
For the necessity, assume that $\cl (W(T))$ is an elliptical disk, where the minor axis
has end points $c+ib/2$, where $b = \|B\|$ and $c \in \{0, \pm 1/2, \pm 1\}$
because the foci are $(0,0), (-1,-1), (1,1), (0,1), (-1,0), (-1,1)$.
Suppose there is a unit vector $x\in H$
such that $\|B^*x\| = \|B\|$.  Consider the $2\times 2$ matrix $T_x=\begin{pmatrix} \langle Ax, x\rangle & \|B^*x\| \\ 0 & 0\end{pmatrix}$. Since $W(T_x)\subseteq W(T)$ and their minor axis have the same length, hence $c=\langle Ax, x\rangle/2$, that is, $|c|\le 1/2$.
It follows that
$c \in\{ 0, \pm 1/2\}$, and hence the foci can only be
$(0,0), (0,1), (-1,0), (-1,1)$.
If the foci are $(0,0)$, then $\cl (W(T))$ is the circular disk centered at 0 with radius
$\|B\|/2$ and thus $\langle Ax, x\rangle=2c=0$.
It implies that $x=(1/\sqrt{2})(u_1+u_2)$, where $u_1\in \ker(A-I)$ and $u_2\in\ker (A+I)$ are unit vectors.
Let $y=(1/\sqrt{2})(u_1-u_2)$,
$\widetilde{x}=x\oplus 0$,
$\widetilde{y}=y\oplus 0$,
and $\widetilde{z}=0\oplus (B^*x/\|B\|)$.
Then $\{\widetilde{x}, \widetilde{y}, \widetilde{z}\}$ is orthonormal and
$$\widetilde{T} \equiv  T|_{\{\widetilde{x}, \widetilde{y}, \widetilde{z}\}} =\begin{pmatrix}
\langle T\widetilde{x}, \widetilde{x}\rangle & \langle T\widetilde{y}, \widetilde{x}\rangle & \langle T\widetilde{z}, \widetilde{x}\rangle\cr
\langle T\widetilde{x}, \widetilde{y}\rangle & \langle T\widetilde{y}, \widetilde{y}\rangle & \langle T\widetilde{z}, \widetilde{y}\rangle\cr
\langle T\widetilde{x}, \widetilde{z}\rangle & \langle T\widetilde{y}, \widetilde{z}\rangle & \langle T\widetilde{z}, \widetilde{z}\rangle\cr
\end{pmatrix}= \begin{pmatrix} 0 & 1 & \|B\|\cr 1 & 0 & 0 \cr 0 & 0 & 0 \cr\end{pmatrix}$$
such that $W(\widetilde{T})$ is not a subset of the circular disk centered at 0 with radius
$\|B\|/2$. Thus, the foci can only be $(0,1), (-1,0), (-1,1)$. By the arguments in the proof of
Theorem \ref{3.1}, we see that $T$ is a unitarily similar to $T_1 \oplus T_2$ as asserted.

Suppose there is no unit vector $x$ such that $\|B^*x\| = \|B\|$.
We can use the Berberian trick in \cite{B} to identify sequences of unit vectors
$\{x_m\}$ in $H$ as an element in
the sequence space. Then identify two sequences $\{x_n\}$ and $\{y_n\}$ if
they differ by a null sequence. We can then use the equivalence classes
in the sequence space to form a Hilbert space
$K$ containing $H$ as the subspace of constant sequences.
Then extend $T$ to an operator
$\tilde T = \begin{pmatrix} \tilde A & \tilde B \cr 0 & 0 \cr\end{pmatrix} \in B(K\oplus K)$,
where $\tilde A, \tilde B\in B(K)$ are the extensions of $A, B\in B(H)$,
i.e., for any sequence $\{x_m\} \in K$, $\tilde A \{x_m\} = \{A x_m\}$ and
$\tilde B \{x_m\} = \{Bx_m\}$. Now, $\tilde B$ will be norm attaining.
By (a),
$\tilde T$ is a direct sum to $\tilde T_0$ and $\tilde T_2$.
Evidently, the norm attaining vector $\{x_m\}$ in $K$ corresponds to
a sequence in $H$ satisfying (b).
\qed

\medskip
\begin{remark}
The necessity of Theorem \ref{3.4} was proved under the assumption that $\cl (W(T))$ is an elliptical disk with foci in $\{0,1,-1\}$.
We conjecture that this assumption is automatic.
\end{remark}

\begin{conjecture}
Let $T\in B(H\oplus H)$ be of the form
\[
T=\begin{pmatrix}A&B\\0&0\end{pmatrix},
\]
where $A=A^*$ and $A^2=I$. If $\cl (W(T))$ is an elliptical disk,
then its foci are contained in $\{0,1,-1\}$.
\end{conjecture}

We are able to prove the conjecture in the case where one of the foci is
$0$; in this case, the other focus must be either $1$ or $-1$.
The general case remains open.

\medskip
 Next we determine the minimum of the numerical radius of $T$ with $A$ such that $A=A^*$ and $A^2=I$
 for a given $B$.  We assume that $A \ne \pm I$ to avoid trivial consideration.
Therefore, we assume that up to unitary similarity, $A = I_p \oplus -I_q$ with
 $p \ge q \ge 1$.  Among other things, suppose $A\in M_k$ and $B\in M_{k, m}$. If $k<m$, by singular value decomposition, we may assume that $B=[B_1 \ 0_{k\times (m-k)}]$ where $B_1\in M_k$, thus $T$ is unitarily similar to $\begin{pmatrix} U^*AU & B_1 \\ 0_k & 0_k\end{pmatrix}\oplus 0_{m-k}$ for some unitary matrix $U\in M_k$. In this case, we have $w(T)=w\Big(\begin{pmatrix} U^*AU & B_1 \\ 0_k & 0_k\end{pmatrix}\Big)$. On the other hand, if $k>m$, let $B'=[B \ 0_{k\times (k-m)}]\in M_k$, then the matrix $T'\equiv\begin{pmatrix} A & B' \\ 0_k & 0_k\end{pmatrix}=T\oplus 0_{k-m}$, and thus $w(T)=w(T')$. Therefore, we only need to consider that both $A$ and $B$ are $k\times k$ matrices.

 For convenience, if $B\in M_k$ has singular values $s_1\ge s_2\ge\cdots\ge s_k\ge 0$, let
\begin{equation}\label{eq1}
S_B=\begin{pmatrix} s_1 & & \\ & \ddots & \\ & & s_k\end{pmatrix}\in M_k \ \ \ \mbox{and} \ \ \ J_k=\begin{pmatrix} & & 1\\ & \rddots & \\ 1 & & \end{pmatrix}\in M_k.
\end{equation}
Among other things, if $A=[a_{ij}]\in M_k$ and $B=[b_{ij}]\in M_k$ satisfy $0\le a_{ij}\le b_{ij}$ for all $1\le i,j\le k$, we write $A\preceq B$. In this case, we have $w(A)\le w(B)$, because a nonnegative matrix attains its numerical radius at a nonnegative unit vector.

Here is our main theorem.

\begin{theorem}\label{t36}
Let $1\le q\le p$, $k=p+q$, $A=I_p\oplus(-I_q)\in M_k$, $B\in M_k$, and
$$\alpha=\min\left\{w\Big(\begin{pmatrix} U^*AU & B\\ 0_k & 0_k\end{pmatrix}\Big) : U\in M_k \mbox{ is unitary}\right\}.$$
Let  $\tilde{ A} = I_p \oplus (-I_p)\in M_{2p}$, $\tilde{B} = B \oplus s_{q+1}I_{p-q}\in M_{2p}$, and
$$\beta=\min\left\{w\Big(\begin{pmatrix} V^* \tilde{A}V & \tilde{B}\\ 0_{2p} & 0_{2p}\end{pmatrix}\Big) : V\in M_{2p} \mbox{ is unitary}\right\}.$$
Then
$$
\alpha=w\Big(\left(\begin{array}{c|c}\begin{array}{ccc}  & & J_q\\ & I_{p-q} & \\ J_q & & \end{array} & S_B
\\ \hline 0_k & 0_k\end{array}\right)\Big)
= w\Big(\begin{pmatrix}
J_{2p} & S_{\tilde{B}} \\ 0_{2p} & 0_{2p}
\end{pmatrix}\Big)=\beta,
$$
where $J_q$, $J_{2p}$, $S_B$, and $S_{\tilde{B}}$ are as in \eqref{eq1}.
\end{theorem}

We need the following lemma to prove Theorem \ref{t36}.

\begin{lemma}\label{l35}
Let $1\le q\le p$, $k=p+q$, $B=\diag(s_1, \dots, s_k)\in M_k$ with $s_1\ge s_2\ge \cdots\ge s_k\ge 0$ and $B_1=\diag(s_{2}, \dots, s_{k-1})\in M_{k-2}$. Then, for any unitary $U\in M_k$, we have
$$
 w\Big(\begin{pmatrix} U^*(I_p\oplus (-I_q))U & B\\ 0_k & 0_k\end{pmatrix}\Big)
\ge   w\Big(\left(\begin{array}{c|c}\begin{array}{cc} 0 & 1\\ 1 & 0\end{array} & \begin{array}{cc} s_1 & 0\\ 0 & s_k\end{array}
\\ \hline 0_2 & 0_2\end{array}\right)\oplus
\begin{pmatrix} {U'}^*(I_{p-1}\oplus (-I_{q-1}))U' & B_1\\ 0_{k-2} & 0_{k-2}\end{pmatrix}\Big) $$
for some unitary matrix $U'\in M_{k-2}$.
\end{lemma}

\begin{proof}
We first show that
$$w\Big(\begin{pmatrix} U^*(I_p\oplus(-I_q))U & B\\ 0_k & 0_k\end{pmatrix}\Big) \ge  w\Big(\left(\begin{array}{c|c}\begin{array}{cc} 0 & 1\\ 1 & 0\end{array} & \begin{array}{cc} s_1 & 0\\ 0 & s_k\end{array}
\\ \hline 0_2 & 0_2\end{array}\right)\Big).$$
Indeed, let $U^*(I_p\oplus(-I_q))U=[a_{ij}]_{i,j=1}^k$ and $A_1=[a_{ij}]_{i,j=2}^k\in M_{k-1}$ be the principal submatrix obtained by deleting the first row and first column of $U^*(I_p\oplus(-I_q))U$. By the interlacing property, there exists a unitary matrix $U_1\in M_{k-1}$ such that $U_1^*A_1U_1=[t_1]\oplus I_{p-1}\oplus (-I_{q-1})$, where $t_1\in [-1, 1]$. Write $U_1^*\diag(s_2, \dots, s_k)U_1=[b_{ij}]_{i,j=1}^{k-1}$, we have $b_{11}\ge s_k$ from the interlacing property. Let $W_1=[1]\oplus U_1\in M_k$ be unitary, then
$$W_1^*(U^*(I_p\oplus(-I_q))U)W_1=\left(\begin{array}{c|c} \begin{array}{cc} a_{11} & a_{12}'\\ a_{21}' & t_1 \end{array} & 0 \\ \hline 0 & \begin{array}{cc} I_{p-1} & \\  & -I_{q-1}\end{array} \end{array}\right)$$
and
$$W_1^*BW_1=\left(\begin{array}{cc|cc} s_1 & 0 &  \cdots & 0 \\ 0 & b_{11} &  \cdots & b_{1, k-1} \\ \hline \vdots & \vdots &  & \vdots\\ 0 & b_{k-1, 1} & \cdots & b_{k-1, k-1}\end{array}\right).$$
Note that the matrix $\left({\textstyle a_{11}\atop  \textstyle a_{21}'} \ {\textstyle a_{12}' \atop \textstyle t_1}\right)$ is unitarily similar to $\diag(1, -1)$ and $b_{11}\ge s_k$, by Lemma \ref{p31} and Corollary \ref{2.4}, we obtain
\begin{eqnarray*}
&& w\Big(\begin{pmatrix} U^*(I_p\oplus(-I_q))U & B\\ 0_k & 0_k\end{pmatrix}\Big)
 \ge    w\Big(\left(\begin{array}{c|c}\begin{array}{cc} a_{11} & a_{12}'  \\ a_{21}' & t_1 \end{array} & \begin{array}{cc} s_1 & 0 \\ 0 & b_{11} \end{array} \\ \hline 0_2 & 0_2\end{array}\right)\Big) \\
&\ge& w\Big(\left(\begin{array}{c|c}\begin{array}{cc} 0 & 1\\ 1 & 0\end{array} & \begin{array}{cc} s_1 & 0\\ 0 & b_{11}\end{array}
\\ \hline 0_2 & 0_2\end{array}\right)\Big)
 \ge  w\Big(\left(\begin{array}{c|c}\begin{array}{cc} 0 & 1\\ 1 & 0\end{array} & \begin{array}{cc} s_1 & 0\\ 0 & s_k\end{array}
\\ \hline 0_2 & 0_2\end{array}\right)\Big)
\end{eqnarray*}
as asserted.

We next show that
$$w\Big(\begin{pmatrix} U^*(I_p\oplus(-I_q))U & B\\ 0_k & 0_k\end{pmatrix}\Big) \ge  w\Big(\begin{pmatrix} {U'}^*(I_{p-1}\oplus (-I_{q-1}))U' & B_1\\ 0_{k-2} & 0_{k-2}\end{pmatrix}\Big)$$
for some unitary matrix $U'\in M_{k-2}$. Indeed, let $A_2=[a_{ij}]_{i,j=1}^{k-1}\in M_{k-1}$ be the leading principal submatrix of $U^*(I_p\oplus(-I_q))U$. By the interlacing property, there exists a unitary matrix $U_2\in M_{k-1}$ such that $U_2^*A_2U_2=[t_2]\oplus I_{p-1}\oplus (-I_{q-1})$, where $t_2\in [-1, 1]$. Write $U_2^*\diag(s_1, \dots, s_{k-1})U_2=[c_{ij}]_{i,j=1}^{k-1}$ and let $C=[c_{ij}]_{i,j=2}^{k-1}\in M_{k-2}$ be the principal submatrix obtained by deleting the first row and first column of $U_2^*\diag(s_1, \dots, s_{k-1})U_2$. Let $W_2=U_2\oplus[1]\in M_k$ be unitary. Then
$$W_2^*(U^*(I_p\oplus(-I_q))U)W_2=\left(\begin{array}{c|c|c}
t_2 & 0 & a_{1k}'' \\ \hline 0 & \begin{array}{cc} I_{p-1} &  \\ & -I_{q-1} \end{array} & 0 \\ \hline a_{k1}'' & 0 & a_{kk}\end{array}\right)$$
and
$$W_2^*BW_2=\left(\begin{array}{c|c|c}
c_{11} & \cdots c_{1, k-1} & 0 \\ \hline
\begin{array}{c} \vdots \\ c_{k-1, 1}\end{array} & C & \begin{array}{c} \vdots \\ 0\end{array} \\ \hline
0 & \cdots \ \ \ 0 & s_k
\end{array}\right),$$
thus $\begin{pmatrix} I_{p-1}\oplus (-I_{q-1}) & C \\ 0_{k-2} & 0_{k-2}\end{pmatrix}$ is a principal submatrix of $\begin{pmatrix} W_2^*(U^*(I_p\oplus(-I_q))U)W_2 & W_2^*BW_2 \\ 0_{k} & 0_{k}\end{pmatrix}$.
Note that $C$ is positive semidefinite with eigenvalues $d_1\ge d_2\ge\cdots\ge d_{k-2}\ge 0$. Let  $D=\diag(d_1, \cdots, d_{k-2})$, then $C=U'D{U'}^*$ for some unitary matrix $U'\in M_{k-2}$. Let $W_3=U'\oplus U'\in M_{2k-4}$ be unitary, we have
$$W_3^*\begin{pmatrix} I_{p-1}\oplus (-I_{q-1}) & C \\ 0_{k-2} & 0_{k-2}\end{pmatrix}W_3=\begin{pmatrix} {U'}^*(I_{p-1}\oplus (-I_{q-1}))U' & D \\ 0_{k-2} & 0_{k-2}\end{pmatrix}.$$ From the interlacing property, we deduce that
$$s_1\ge d_1\ge s_2\ge d_2\ge \cdots\ge s_{k-2}\ge d_{k-2}\ge s_{k-1},$$
thus $D \succeq B_1=\diag(s_2, \dots, s_{k-1})$.  Corollary \ref{2.4} yields that
$$w\Big(\begin{pmatrix} {U'}^*(I_{p-1}\oplus (-I_{q-1}))U' & D \\ 0_{k-2} & 0_{k-2}\end{pmatrix}\Big) \ge w\Big(\begin{pmatrix} {U'}^*(I_{p-1}\oplus (-I_{q-1}))U' & B_1 \\ 0_{k-2} & 0_{k-2}\end{pmatrix}\Big).$$
Therefore, we obtain
\begin{eqnarray*}
& & w\Big(\begin{pmatrix} U^*(I_p\oplus(-I_q))U & B\\ 0_k & 0_k\end{pmatrix}\Big)
\ge w\Big(\begin{pmatrix} I_{p-1}\oplus (-I_{q-1}) & C \\ 0_{k-2} & 0_{k-2}\end{pmatrix}\Big) \\
&=& w\Big(\begin{pmatrix} {U'}^*(I_{p-1}\oplus (-I_{q-1}))U' & D \\ 0_{k-2} & 0_{k-2}\end{pmatrix}\Big)
\ge w\Big(\begin{pmatrix} {U'}^*(I_{p-1}\oplus (-I_{q-1}))U' & B_1\\ 0_{k-2} & 0_{k-2}\end{pmatrix}\Big)\end{eqnarray*}
as asserted. This completes the proof.
\end{proof}

We are now ready to prove Theorem \ref{t36}.

\begin{proof}[Proof of Theorem \ref{t36}]
 First we consider $\alpha$.
Let $B=Q_0S_BQ_1^*$ be the singular value decomposition of $B$ and $W_0=Q_0\oplus Q_1\in M_{2k}$ be unitary, then, for any unitary $U\in M_k$, we have
$$W_0^*\begin{pmatrix} U^*AU & B\\ 0_k & 0_k\end{pmatrix}W_0=\begin{pmatrix} (UQ_0)^*A(UQ_0) & S_B\\ 0_k & 0_k\end{pmatrix}.$$
Since $UQ_0\in M_k$ is also unitary, we deduce that
\begin{equation}\label{eq3}
\alpha=\min\left\{w\Big(\begin{pmatrix} U^*AU & S_B\\ 0_k & 0_k\end{pmatrix}\Big) : U\in M_k \mbox{ is unitary}\right\}\le
w\Big(\left(\begin{array}{c|c}\begin{array}{ccc}  & & J_q\\ & I_{p-q} & \\ J_q & & \end{array} & S_B
\\ \hline 0_k & 0_k\end{array}\right)\Big).
\end{equation}
Therefore, we may assume that $B=S_B=\diag(s_1, \dots, s_k)$. Let $B_j=\diag(s_{j+1}, \dots, s_{k-j})\in M_{k-2j}$ for $j=1, \dots, q$.

Now, for any unitary $U\in M_k$, apply Lemma \ref{l35} repeatedly, we obtain
\begin{eqnarray*}
& & w\Big(\begin{pmatrix} U^*AU & B\\ 0_k & 0_k\end{pmatrix}\Big)   \\
& \ge & w\Big(\left(\begin{array}{c|c}\begin{array}{cc} 0 & 1\\ 1 & 0\end{array} & \begin{array}{cc} s_1 & 0\\ 0 & s_k\end{array}
\\ \hline 0_2 & 0_2\end{array}\right)\oplus
\begin{pmatrix} U_1^*(I_{p-1}\oplus (-I_{q-1}))U_1 & B_1\\ 0_{k-2} & 0_{k-2}\end{pmatrix}\Big) \\
& \ge & w\Big(\left(\begin{array}{c|c}\begin{array}{cc} 0 & 1\\ 1 & 0\end{array} & \begin{array}{cc} s_1 & 0\\ 0 & s_k\end{array}
\\ \hline 0_2 & 0_2\end{array}\right)\oplus
\left(\begin{array}{c|c}\begin{array}{cc} 0 & 1\\ 1 & 0\end{array} & \begin{array}{cc} s_2 & 0\\ 0 & s_{k-1}\end{array}
\\ \hline 0_2 & 0_2\end{array}\right) \oplus
\begin{pmatrix} U_2^*(I_{p-2}\oplus (-I_{q-2}))U_2 & B_2\\ 0_{k-4} & 0_{k-4}\end{pmatrix}\Big) \\
& \ge & \cdots \\
& \ge & w\Big(\left(\begin{array}{c|c}\begin{array}{cc} 0 & 1\\ 1 & 0\end{array} & \begin{array}{cc} s_1 & 0\\ 0 & s_k\end{array}
\\ \hline 0_2 & 0_2\end{array}\right)\oplus \cdots\oplus
\left(\begin{array}{c|c}\begin{array}{cc} 0 & 1\\ 1 & 0\end{array} & \begin{array}{cc} s_q & 0\\ 0 & s_{k-q+1}\end{array}
\\ \hline 0_2 & 0_2\end{array}\right) \oplus
\begin{pmatrix} U_q^*I_{p-q}U_q & B_q\\ 0_{k-2q} & 0_{k-2q}\end{pmatrix}\Big) \\
& = & w\Big(\left(\begin{array}{c|c}\begin{array}{cc} 0 & 1\\ 1 & 0\end{array} & \begin{array}{cc} s_1 & 0\\ 0 & s_k\end{array}
\\ \hline 0_2 & 0_2\end{array}\right)\oplus \cdots\oplus
\left(\begin{array}{c|c}\begin{array}{cc} 0 & 1\\ 1 & 0\end{array} & \begin{array}{cc} s_q & 0\\ 0 & s_{p+1}\end{array}
\\ \hline 0_2 & 0_2\end{array}\right) \oplus
\begin{pmatrix} I_{p-q} & B_q\\ 0_{p-q} & 0_{p-q}\end{pmatrix}\Big),
\end{eqnarray*}
where $U_j\in M_{k-2j}$ is unitary for $j=1, \dots, q$. Note that the matrix
$$\left(\begin{array}{c|c}\begin{array}{ccc}  & & J_q\\ & I_{p-q} & \\ J_q & & \end{array} & S_B
\\ \hline 0_k & 0_k\end{array}\right)$$
is permutationally similar to
$$\left(\begin{array}{c|c}\begin{array}{cc} 0 & 1\\ 1 & 0\end{array} & \begin{array}{cc} s_1 & 0\\ 0 & s_k\end{array}
\\ \hline 0_2 & 0_2\end{array}\right)\oplus \cdots\oplus
\left(\begin{array}{c|c}\begin{array}{cc} 0 & 1\\ 1 & 0\end{array} & \begin{array}{cc} s_q & 0\\ 0 & s_{p+1}\end{array}
\\ \hline 0_2 & 0_2\end{array}\right) \oplus
\begin{pmatrix} I_{p-q} & B_q\\ 0_{p-q} & 0_{p-q}\end{pmatrix},$$
hence we conclude that
$$w\Big(\begin{pmatrix} U^*AU & B\\ 0_k & 0_k\end{pmatrix}\Big) \ge w\Big(\left(\begin{array}{c|c}\begin{array}{ccc}  & & J_q\\ & I_{p-q} & \\ J_q & & \end{array} & S_B
\\ \hline 0_k & 0_k\end{array}\right)\Big)$$
for any unitary $U\in M_k$, that is,
$$\alpha\ge
w\Big(\left(\begin{array}{c|c}\begin{array}{ccc}  & & J_q\\ & I_{p-q} & \\ J_q & & \end{array} & S_B
\\ \hline 0_k & 0_k\end{array}\right)\Big) \ \ \mbox{ or } \ \
\alpha=
w\Big(\left(\begin{array}{c|c}\begin{array}{ccc}  & & J_q\\ & I_{p-q} & \\ J_q & & \end{array} & S_B
\\ \hline 0_k & 0_k\end{array}\right)\Big)$$
by \eqref{eq3}.

Now we turn to $\beta$.
From the proof on $\alpha$, we have
\begin{eqnarray*}
\alpha & = & w\Big(\left(\begin{array}{c|c}\begin{array}{ccc}  & & J_q\\ & I_{p-q} & \\ J_q & & \end{array} & S_B
\\ \hline 0_k & 0_k\end{array}\right)\Big) \\
&= &  w\Big(\left(\begin{array}{c|c}\begin{array}{cc} 0 & 1\\ 1 & 0\end{array} & \begin{array}{cc} s_1 & 0\\ 0 & s_k\end{array}
\\ \hline 0_2 & 0_2\end{array}\right)\oplus \cdots\oplus
\left(\begin{array}{c|c}\begin{array}{cc} 0 & 1\\ 1 & 0\end{array} & \begin{array}{cc} s_q & 0\\ 0 & s_{p+1}\end{array}
\\ \hline 0_2 & 0_2\end{array}\right) \oplus
\begin{pmatrix} I_{p-q} & B_q\\ 0_{p-q} & 0_{p-q}\end{pmatrix}\Big)\\
&=& \max\left\{\max_{1\le j\le q}w\Big(\left(\begin{array}{c|c}\begin{array}{cc} 0 & 1\\ 1 & 0\end{array} & \begin{array}{cc} s_j & 0\\ 0 & s_{k-j+1}\end{array}
\\ \hline 0_2 & 0_2\end{array}\right)\Big), \ \ \max_{q+1\le j\le p} w\Big(\begin{pmatrix} 1 & s_{j}\\ 0 & 0\end{pmatrix}\Big)\right\}.
\end{eqnarray*}
Since $s_{q+1}\ge s_{q+2}\ge \cdots\ge s_p$, we have
$w\Big(\begin{pmatrix} 1 & s_{j}\\ 0 & 0\end{pmatrix}\Big) \le w\Big(\begin{pmatrix} 1 & s_{q+1}\\ 0 & 0\end{pmatrix}\Big)=(1+\sqrt{1+s_{q+1}^2})/2$ for all $q+2\le j\le p$. Therefore, we obtain
\begin{equation}\label{ee1}
\alpha=\max\left\{\max_{1\le j\le q}w\Big(\left(\begin{array}{c|c}\begin{array}{cc} 0 & 1\\ 1 & 0\end{array} & \begin{array}{cc} s_j & 0\\ 0 & s_{k-j+1}\end{array}
\\ \hline 0_2 & 0_2\end{array}\right)\Big), \ \ \frac{1+\sqrt{1+s_{q+1}^2}}{2}\right\}.
\end{equation}
On the other hand, since $\tilde{B} = B \oplus s_{q+1}I_{p-q}$, from \eqref{eq1}, we have
$$S_{\tilde{B}}=\diag(s_1, \dots, s_q)\oplus s_{q+1}I_{p-q}\oplus \diag(\underbrace{s_{q+1}, \dots, s_p}_{p-q})\oplus \diag(\underbrace{s_{p+1}, \dots, s_k}_q).$$
From we have proved before, we obtain
$$
\beta  \equiv  \min\left\{w\Big(\begin{pmatrix} V^* \tilde{A}V & \tilde{B}\\ 0_{2p} & 0_{2p}\end{pmatrix}\Big) : V\in M_{2p} \mbox{ is unitary}\right\}
=  \ w\Big(\left(\begin{array}{cc} J_{2p} & S_{\tilde{B}}
\\ 0_{2p} & 0_{2p}\end{array}\right)\Big). $$
Note that $\left(\begin{array}{cc} J_{2p} & S_{\tilde{B}}
\\ 0_{2p} & 0_{2p}\end{array}\right)$ is permutationally similar to
$$\Bigg(\sum_{j=1}^q\oplus\left(\begin{array}{c|c}\begin{array}{cc} 0 & 1\\ 1 & 0\end{array} & \begin{array}{cc} s_j & 0\\ 0 & s_{k-j+1}\end{array}
\\ \hline 0_2 & 0_2\end{array}\right)\Bigg) \oplus \Bigg(\sum_{j=q+1}^p\oplus\left(\begin{array}{c|c}\begin{array}{cc} 0 & 1\\ 1 & 0\end{array} & \begin{array}{cc} s_{q+1} & 0\\ 0 & s_{j}\end{array}
\\ \hline 0_2 & 0_2\end{array}\right)\Bigg).$$
In particular, the matrix
$$\left(\begin{array}{c|c}\begin{array}{cc} 0 & 1\\ 1 & 0\end{array} & \begin{array}{cc} s_{q+1} & 0\\ 0 & s_{q+1}\end{array}
\\ \hline 0_2 & 0_2\end{array}\right)$$ is unitarily similar to $\begin{pmatrix} 1 & s_{q+1} \\ 0 & 0\end{pmatrix}\oplus \begin{pmatrix} -1 & s_{q+1} \\ 0 & 0\end{pmatrix}$, this yields that
$$w\Big(\left(\begin{array}{c|c}\begin{array}{cc} 0 & 1\\ 1 & 0\end{array} & \begin{array}{cc} s_{q+1} & 0\\ 0 & s_{q+1}\end{array}
\\ \hline 0_2 & 0_2\end{array}\right)\Big)
=w\Big( \begin{pmatrix} 1 & s_{q+1} \\ 0 & 0\end{pmatrix}\Big)
=  \frac{1+\sqrt{1+s_{q+1}^2}}{2}.$$
Moreover, since $s_{q+1}\ge s_{q+2}\ge \cdots\ge s_p$, by Corollary \ref{2.4}, we have
$$w\Big(\left(\begin{array}{c|c}\begin{array}{cc} 0 & 1\\ 1 & 0\end{array} & \begin{array}{cc} s_{q+1} & 0\\ 0 & s_{j}\end{array}
\\ \hline 0_2 & 0_2\end{array}\right)\Big) \le w\Big(\left(\begin{array}{c|c}\begin{array}{cc} 0 & 1\\ 1 & 0\end{array} & \begin{array}{cc} s_{q+1} & 0\\ 0 & s_{q+1}\end{array}
\\ \hline 0_2 & 0_2\end{array}\right)\Big) \ \  \mbox{for all }  q+2\le j\le p.$$
It follows that
$$w\Big(\sum_{j=q+1}^p\oplus\left(\begin{array}{c|c}\begin{array}{cc} 0 & 1\\ 1 & 0\end{array} & \begin{array}{cc} s_{q+1} & 0\\ 0 & s_{j}\end{array}
\\ \hline 0_2 & 0_2\end{array}\right)\Big)= w\Big(\left(\begin{array}{c|c}\begin{array}{cc} 0 & 1\\ 1 & 0\end{array} & \begin{array}{cc} s_{q+1} & 0\\ 0 & s_{q+1}\end{array}
\\ \hline 0_2 & 0_2\end{array}\right)\Big) = \frac{1+\sqrt{1+s_{q+1}^2}}{2}.$$
Hence we obtain
\begin{eqnarray*}
\beta &= & w\Big(\left(\begin{array}{cc} J_{2p} & S_{\tilde{B}}
\\ 0_{2p} & 0_{2p}\end{array}\right)\Big) \\
& = & \max\left\{ w\Big(\sum_{j=1}^q\oplus\left(\begin{array}{c|c}\begin{array}{cc} 0 & 1\\ 1 & 0\end{array} & \begin{array}{cc} s_j & 0\\ 0 & s_{k-j+1}\end{array}
\\ \hline 0_2 & 0_2\end{array}\right)\Big), w\Big( \sum_{j=q+1}^p\oplus\left(\begin{array}{c|c}\begin{array}{cc} 0 & 1\\ 1 & 0\end{array} & \begin{array}{cc} s_{q+1} & 0\\ 0 & s_{j}\end{array}
\\ \hline 0_2 & 0_2\end{array}\right)\Big)  \right\}\\
& = & \max\left\{\max_{1\le j\le q}w\Big(\left(\begin{array}{c|c}\begin{array}{cc} 0 & 1\\ 1 & 0\end{array} & \begin{array}{cc} s_j & 0\\ 0 & s_{k-j+1}\end{array}
\\ \hline 0_2 & 0_2\end{array}\right)\Big), \ \ \frac{1+\sqrt{1+s_{q+1}^2}}{2}\right\} = \alpha
\end{eqnarray*}
from \eqref{ee1}, this completes the proof.
\end{proof}

\medskip
We can extend Theorem \ref{t36} to infinite dimensional operators using the fact
that $B(H)$ is a $C$*-algebra of real rank zero; see \cite{BP}.
The positive semidefinite operator $B$ in its integral representation
$\int_{\lambda} \lambda dE(\lambda)$ can be approximated by a sequence of
 ``diagonal'' operators $\{B_m\}$ each has a finite spectrum.
By a suitable choice of unitary basis, we may consider operators of the form
$T_m = \begin{pmatrix} A_m & B_m \cr 0 & 0 \cr\end{pmatrix}$
with $A_m = \begin{pmatrix} 0 & 0 & J \cr 0 & I & 0 \cr J & 0 & 0 \cr\end{pmatrix}$,
where $J$ is the infinite dimensional extension of $J_k$, say, $J(e_n) = e_{-n}$
if $J \in B(H)$ with $H =\ell_2(\IZ)$, and $J(f(x)) = f(-x)$ if $J \in B(H)$ with $H = L_2(\IR)$.
Then one can show that $w(T_m)$ is minimum among operators of the form
$\begin{pmatrix} UA_mU^* & B_m \cr 0 & 0 \end{pmatrix}$ and  the conclusion also holds
for the limit.

\section{Matrix powers and constant norm attaining vectors}

In this section, we study matrices for which there exists a unit vector
$x$ satisfying
\[
\|A^k x\| = \|A^k\| = \|A\| \  \  \mbox{ for all } k = 1,2,\ldots.
\]
Such a vector will be called a \emph{common norm attaining vector} for the powers of $A$.
This problem is motivated by a simple observation.
If $A \in M_n$ is an idempotent, i.e.,  $A^2 = A$, then the sequence $\{\|A^k\|\}_{k \ge 1}$ is
constant, and there exists
a unit vector attaining these norms simultaneously. This leads naturally to the following question.

\begin{problem} \label{prob}
For which matrices $A \in M_n$ does there exist a unit vector $x$ such that
\[
\|A^k x\| = \|A^k\| = \|A\| \  \  \mbox{ for all } k = 1,2,\ldots?
\]
\end{problem}

Observe that if $\|A\| = \|A^2\| > 0$, then
\[
\|A\| = \|A^2\| \le \|A\|^2 ,
\]
which implies that $\|A\| \ge 1$.

The cases $\|A\| = 1$ and $n = 2$ were studied in \cite{GWW} and \cite{P}, respectively, leading to the following results.
The next two results are from \cite[Corollary 2.6 and Proposition 3.2 (b)]{GWW}.

\begin{proposition}\label{prop41}
For an $n \times n$ matrix $A$, the following are equivalent:
\begin{itemize}
\item[\rm (a)] $\|A^n\| = \|A\| = 1$.
\item[\rm (b)] $\|A^k\| = 1$ for all $k \ge 1$.
\item[\rm (c)] There exists a unit vector $x$ such that $\|A^k x\| = \|A^k\| = \|A\| = 1$ for all $k \ge 1$.
\item[\rm (d)] $A$ is unitarily similar to $[a] \oplus B$ with $\|B\| \le |a| = 1$.
\end{itemize}
\end{proposition}

\begin{proposition}\label{prop42}
For a $2 \times 2$ matrix $A$, the following are equivalent:
\begin{itemize}
\item[\rm (a)] There exists a unit vector $x$ such that $\|A^k x\| = \|A^k\| = \|A\| > 1$ for all $k \ge 1$;
\item[\rm (b)] $A$ is unitarily similar to
\[
\lambda \begin{pmatrix}
1 & b \\
0 & 0
\end{pmatrix},
\]
where $|\lambda| = 1$ and $b \neq 0$.
\end{itemize}
\end{proposition}

Regarding Problem \ref{prob}, the next result shows that the matrix
$T=\begin{pmatrix}
U & B \\
0 & 0
\end{pmatrix}\in M_n$,
where $U$ is unitary, satisfies the condition in Problem \ref{prob}.

\begin{theorem}\label{t410}
Let $T\in M_n$. Then the following are equivalent.
\begin{itemize}
  \item[\rm (a)] $T$ is unitarily similar to a matrix of the form
  \[
  \begin{pmatrix}
    U & B\\
    0 & 0_{n-r}
  \end{pmatrix},
  \]
  where $U\in M_r$, $0\le r\le n$, is unitary.
  \item[\rm (b)] $T^*T={T^k}^*T^k$ for all $k\ge 1$.
  \item[\rm (c)] $T^* T={T^2}^* T^2$.
\end{itemize}
In this case, the following statements hold:
\begin{itemize}
  \item[\rm (i)] $\|T^kx\|=\|T^k\|=\|T\|$ for some unit vector $x\in\IC^n$ and for all $k\ge 1$.
    \item[\rm (ii)] If $T$ is nonzero, then $\|T\|=\sqrt{1+\|B\|^2}$.
  \item[\rm (iii)] If $\lambda$ is a nonzero eigenvalue of $T$, then $|\lambda|=1$.
\end{itemize}
\end{theorem}

\begin{proof}
(b) $\Rightarrow$ (c) is trivial.

(c) $\Rightarrow$ (b). Assume that $T^*T={T^2}^*T^2$. Then, for $k\ge 2$,
$${T^k}^*T^k= {T^{k-2}}^*{T^2}^*T^2T^{k-2}={T^{k-2}}^*T^*TT^{k-2}={T^{k-1}}^*T^{k-1}.$$
Continue this process, we infer that ${T^k}^* T^k={T^{k-1}}^*T^{k-1}={T^{k-2}}^*T^{k-2}= \cdots =T^*T$ as required.

(b) $\Rightarrow$ (a). By the Schur’s triangularization theorem (cf. \cite[Theorem 2.3.1]{HJ}), $T$ is unitarily similar to an upper triangular matrix of the form
$
\begin{pmatrix}
U & B\\
0 & N
\end{pmatrix},
$
where $U$ is an $r \times r$ ($0\le r\le n$) nonsingular upper triangular matrix and $N\in M_{n-r}$ is upper triangular with $\sigma(N)=\{0\}$. We first check that $N=0$. Indeed, since $N$ is nilpotent, then $N^n=0$, and thus
$T^n=
\begin{pmatrix}
U^n & *\\
0 & 0
\end{pmatrix}$. Since $U$ is nonsingular, it follows that
$$
\rank(T^n)=\rank(U^n)=\rank(U).
$$
On the other hand, $\rank(X)=\rank (X^*X)$ for any matrix $X\in M_n$. Therefore, condition (b) yields that
$$\rank(T)=\rank(T^*T)=\rank({T^n}^* T^n)=\rank(T^n)=\rank (U).$$
Hence $\rank(N)=\rank(T)-\rank(U)=0$, or $N=0$ as asserted.

It remains to show that $U$ is unitary for $r\ge 1$. If $r=0$, then $T=N=0$ as asserted form. Therefore, we assume that $r\ge 1$. The equality $T^*T={T^2}^* T^2$ yields that
\[
\begin{pmatrix}
U^*U & U^*B\\
B^*U & B^*B
\end{pmatrix}
=
\begin{pmatrix}
{U^2}^*U^2 & {U^2}^*UB\\
B^*U^*U^2 & B^*U^*UB
\end{pmatrix}.
\]
Comparing the $(1,1)$-blocks yields
$U^*U = {U^2}^*U^2$.
Since $U$ is invertible, we multiply by $(U^*)^{-1}$ on the left and $U^{-1}$ on the right to obtain $I_r = U^*U$.
Hence $U$ is unitary.

(a) $\Rightarrow$ (b). After a unitary similarity, we may assume that
$
T=\begin{pmatrix}
U & B\\
0 & 0_{n-r}
\end{pmatrix},
$
where $U\in M_r$ is unitary.  Let $V=U\oplus I_{n-r}\in M_n$ be unitary, then, for $k\ge 1$, we have
$$T^k=\begin{pmatrix}
U^k & U^{k-1}B \\
0 & 0
\end{pmatrix}=\begin{pmatrix}
U^{k-1} & 0 \\
0 & I_{n-r}
\end{pmatrix}\begin{pmatrix}
U & B \\
0 & 0
\end{pmatrix}=V^{k-1}T,$$
and thus
$${T^k}^*T^k=T^*{V^{k-1}}^*V^{k-1}T=T^*T$$
as required.

Moreover, since $V$ is unitary, it follows that $\|T^k\|=\|V^{k-1}T\|=\|T\|$ for all $k\ge1$. Among
other things, since $T$ is a finite matrix, then $\|Tx\|=\|T\|$ for some unit vector $x\in\IC^n$,
thus $\|T^kx\|=\|V^{k-1}Tx\|=\|Tx\|=\|T\|=\|T^k\|$ for all $k\ge 1$, hence (i) holds.

Next, it is easily seen that
$$
TT^*=\begin{pmatrix}
I_r+BB^* & 0\\
0 & 0
\end{pmatrix},$$
hence $\|T\|^2=\|T^*\|^2=\|TT^*\|=\|I_r+BB^*\|=1+\|BB^*\|=1+\|B^*\|^2=1+\|B\|^2$, this proves (ii).

Finally, (iii) follows from $\sigma(T)\subseteq\sigma(U)\cup\{0\}$ and $U$ is unitary. This completes the proof.
\end{proof}

We aim to formulate a necessary and sufficient condition for a nonzero matrix
$A\in M_n$  satisfying
\[
\|A^k x\| = \|A^k\| = \|A\|  \quad \text{for all } k \ge 1 \mbox{ and for some unit vextor } x\in\IC^n.
\]
The following theorem is the main result of this section, it gives a complete solution of Problem \ref{prob}. Note that the case $\|A\|=1$ has been solved by Proposition \ref{prop41}. Therefore, we need only to consider the case $\|A\|>1$.

\begin{theorem}\label{t44}
Let $A\in M_n$ with $\|A\|>1$, and let $m$ be the degree of the minimal polynomial of $A$.
The following conditions are equivalent:
\begin{itemize}
\item[\rm (a)] There exists a unit vector $x\in \mathbb{C}^n$ such that
\[
\|A^k x\| = \|A^k\| = \|A\|  \quad \mbox{ for all } k = 1,2,\ldots,m+1.
\]

\item[\rm (b)] $A$ is unitarily similar to a matrix of the form
\[
\begin{pmatrix}
T & D\\
0 & C
\end{pmatrix},
\]
where $T\in M_r$, $2 \le r \le m$, satisfies
\[
T^*T = {T^2}^*T^2 \ \ \mbox{ and } \ \
\|A^k\| = \|T^k\| \quad \mbox{ for all }  k = 1,2,\ldots,m+1.
\]
In particular, if $r=2$, then $D=0$.
\end{itemize}
In this case, $A$ is singular.
\end{theorem}

For the proof of Theorem \ref{t44}, we need the next lemma.

In the following, we use $\{e_1, \dots, e_n\}$ to denote the standard basis of $\IC^n$. Moreover, for any matrix $X\in M_n$, if $\|Xx\|=\|X\|$ for some unit vector $x\in\IC^n$, then $X^*Xx=\|X\|^2x$, because $\langle(\|X\|^2I_n-X^*X)x,x\rangle=0$ and $\|X\|^2I_n-X^*X$ is positive semidefinite.

\begin{lemma}\label{l46}
For $n\ge 2$, let $T$ be an $n\times n$ matrix of the form
\[
\begin{pmatrix}
t_{11} & t_{12} & \cdots & t_{1, n-1} & t_{1n}\\
t_{21} & t_{22} & \cdots & t_{2, n-1} & t_{2n}\\
0      & t_{32} & \cdots & t_{3, n-1} & t_{3n}\\
\vdots & \vdots & \ddots & \vdots     & \vdots \\
0      & 0      & \cdots & t_{n, n-1} & t_{nn}
\end{pmatrix},
\]
where $t_{i, i-1}\neq 0$ for $i=2,3,\ldots,n$.
If $\| T^{k}e_{1} \|= \| T^{k} \|= \| T \|$ for $k=1,2,\ldots, n+1$, then
\[
T^*T = {T^2}^*T^{2}.
\]
\end{lemma}

\begin{proof}
Note that $Te_j\in\bigvee\{e_1, \dots, e_{j+1}\}$ for $j=1, \dots, n-1$, it follows that
$$T^2e_j\in T(\bigvee\{e_1, \dots, e_{j+1}\})= \bigvee\{Te_1, \dots, Te_{j+1}\}\subseteq \bigvee\{e_1, \dots, e_{j+2}\}.$$
Continue this process, we have $T^ke_j\in \bigvee\{e_1, \dots, e_{j+k}\}$. In particular,
\begin{equation}\label{eq41}
T^ke_1\in \bigvee\{e_1, \dots, e_{k+1}\} \ \  \mbox{ for } 1\le k\le n-1,
\end{equation}
and thus
\begin{equation}\label{eq42}
\langle T^je_1, e_{k+1}\rangle=0 \ \  \mbox{ for } 0\le j<k\le n-1.
\end{equation}
Moreover, for $k=1,2,\ldots, n-1$, since
$$\langle T^ke_1, e_{k+1}\rangle= \langle T(T^{k-1}e_1), e_{k+1}\rangle = \langle \sum_{j=1}^k\langle T^{k-1}e_1, e_j\rangle Te_j, e_{k+1}\rangle= \langle T^{k-1}e_1, e_k\rangle  \langle Te_k, e_{k+1}\rangle,$$
we infer that
\begin{equation}\label{eq45}
\langle T^k e_1,e_{k+1}\rangle=\langle Te_1,e_{2}\rangle \langle T e_2,e_{3}\rangle \cdots \langle T e_k,e_{k+1}\rangle=t_{21}\cdots t_{k+1, k}\neq 0 \ \ \mbox{ for } k=1,2,\ldots, n-1.
\end{equation}
Therefore, by \eqref{eq41}, \eqref{eq42} and \eqref{eq45}, we obtain successively $e_{k+1}\in \bigvee\{e_1,Te_1,\ldots,T^k e_1\}$ and
$$
\bigvee\{e_1,Te_1,\ldots,T^k e_1\}
=
\bigvee\{e_1,e_2,\ldots,e_{k+1}\}\ \ \mbox{ for } k=0,1,\ldots,n-1.
$$
In particular, this implies that $\{e_1,Te_1, \ldots, T^{n-1}e_1\}$ is a basis of $\mathbb{C}^n$. Therefore, it suffices to prove that
\begin{equation}\label{eq43}
T^*T(T^k e_1) = {T^2}^*T^{2}(T^k e_1)\ \ \mbox{ for all }  k=0,1,\ldots,n-1.
\end{equation}

We proceed by induction on $k$.

\medskip
\noindent
\textbf{(i) The case $k=0$.}
Since $\|Te_1\|=\|T\|=\|T^2\|=\|T^2 e_1\|$, hence
\[
T^*T e_1 = \|T\|^2 e_1 = \|T^2\|^2 e_1 = {T^2}^*T^2 e_1.
\]

\medskip
\noindent
\textbf{(ii) The case $1\le k\le n-1$.}
Assume that \eqref{eq43} holds for $0,\ldots,k-1$, that is,
\begin{equation}\label{eq44}
T^*(T^*T-I_n)T(T^je_1)=0 \ \ \mbox{ for all } j=0,1,\ldots,k-1.
\end{equation}
We want to show that $T^*T(T^k e_1) = {T^2}^*T^{2}(T^k e_1)$. Indeed, Let
$$\mu={T^2}^*T^2(T^k e_1)-T^*T(T^k e_1) = T^*(T^*T-I_n)T(T^ke_1).$$ Then, for $0\le j\le k-1$, \eqref{eq44} yields that
$$\langle \mu, T^je_1\rangle= \langle  T^*(T^*T-I_n)T(T^ke_1), T^je_1\rangle= \langle T^ke_1,  T^*(T^*T-I_n)T(T^je_1)\rangle= \langle T^ke_1, 0\rangle=0.$$
Moreover, by assumption, we have $\|T^{k+2}e_1\|=\|T^{k+2}\|=\|T\|=\|T^{k+1}\|=\|T^{k+1}e_1||$, we infer that
$${T^{k+2}}^*T^{k+2}e_1=\|T^{k+2}\|^2e_1=\|T\|^2e_1=\|T^{k+1}\|^2e_1={T^{k+1}}^*T^{k+1}e_1,$$
or, ${T^{k+1}}^*(T^*T-I_n)T^{k+1}e_1=0$. Therefore, for $k\le j\le n-1$, we obtain
$$\langle \mu, T^je_1\rangle= \langle  T^*(T^*T-I_n)T(T^ke_1), T^je_1\rangle= \langle {T^{k+1}}^*(T^*T-I_n)T^{k+1}e_1, T^{j-k}e_1)\rangle= \langle 0, T^{j-k}e_1 \rangle=0.$$
Since $\{e_1, Te_1, \dots, T^{n-1}e_1\}$ is a basis of $\IC^n$ and $\langle \mu, T^je_1\rangle=0$ for all $0\le j\le n-1$,
we deduce that $\mu=0$, and hence
\[
T^*T (T^k e_1) = {T^2}^*T^2 (T^k e_1).
\]
This completes the proof.
\end{proof}

We are now ready to prove Theorem \ref{t44}.

\begin{proof}[Proof of Theorem \ref{t44}]
(b) $\Rightarrow$ (a).
Assume that
$
A = \begin{pmatrix}
T & D\\
0 & C
\end{pmatrix},
$
where $T$ satisfies the conditions in (b). Theorem \ref{t410} (i) yields that
\[
\|T^kx\|=\|T^k\|=\|T\| \ \ \mbox{ for some unit vector } x\in\IC^r \mbox{ and for all } k\ge 1.
\]
Then
\begin{equation}\label{eq48}
A^k = \begin{pmatrix}
T^k & *\\
0 & C^k
\end{pmatrix}
\quad \text{for all } k\ge 1.
\end{equation}
Set $\widetilde{x} =\left({\textstyle x\atop  \textstyle 0}\right) \in \mathbb{C}^n$. Then $A^k \widetilde{x} = \Big({\textstyle T^kx\atop  \textstyle 0}\Big)$, and for $1 \le k \le m+1$,
\[
\|A^k \widetilde{x}\| = \|T^k x\| = \|T^k\|  = \|A^k\| = \|T\|= \|A\|.
\]
Hence (a) holds.

\medskip

(a) $\Rightarrow$ (b).
Let
\[
K=\bigvee \{x, Ax, \ldots, A^{m-1}x\}\ \ \mbox{ and }  \ \ r = \dim K.
\]
Since $m$ is the degree of the minimal polynomial of $A$, thus $K$ is an invariant subspace of $A$ and $r\le m$. We first show that $r \ge 2$.
If $r=1$, then $Ax = \lambda x$ for some $\lambda\in \mathbb{C}$.
Hence
\[
|\lambda| = \|Ax\| = \|A\| = \|A^2 x\| = |\lambda|^2,
\]
which implies $|\lambda| = 1$, contradicting $\|A\| > 1$.
Thus $r \ge 2$. Since $r=\dim K$,
the vectors $\{x, Ax, \ldots, A^{r-1}x\}$ are linearly independent.
Extend them to an orthonormal basis $\{u_1, u_2, \ldots, u_n\}$ of $\mathbb{C}^n$ such that
\[
\bigvee \{x, Ax, \ldots, A^{k-1}x\}
= \bigvee \{u_1, u_2, \ldots, u_k\}, \quad k=1,\ldots,r,
\]
where $u_1 = x$. Let $U=[u_1 \ \dots \ u_n]\in M_n$ be unitary. Since $K=\bigvee \{x, Ax, \ldots, A^{r-1}x\}=\bigvee \{u_1, u_2, \ldots, u_r\}$ is invariant for $A$, $U^*AU$ has the form
\[
\widetilde{A}\equiv U^*AU = \begin{pmatrix}
T & D\\
0 & C
\end{pmatrix},
\]
where $T=[t_{ij}]\in M_r$ is upper Hessenberg with $t_{i,i-1}\ne 0$ for $i=2,\ldots,r$.
Note that $Ue_1=u_1=x$. Let $e_1' = [1\;0\;\cdots\;0]^t \in \mathbb{C}^r$.
By assumption $\|A^kx\| =\|A^k\|$, $1\le k\le m+1$, we have
\[
\|T^k\| \le  \|\widetilde{A}^k\|=\|A^k\| =\|A^k x\|=\|U\widetilde{A}^kU^*(Ue_1)\|=
\|\widetilde{A}^k e_1\| = \|T^k e_1'\| \le \|T^k\|.
\]
Thus the inequalities above are actually equalities throughout. This implies that
\[
\|A^k\| = \|T^k e_1'\| = \|T^k\|
\]
for $k=1,2,\ldots,m+1$. In particular, $\|T\|=\|A\|=\|A^k\|=\|T^k\|$ for all $1\le k\le m+1$.
That is,
\[
\|T^k e_1'\| = \|T^k\| = \|T\| \quad \text{for } k=1,\ldots,r+1.
\]
By Lemma \ref{l46}, we obtain $T^*T = T^{*2}T^2$
as required.

\medskip

In particular, if $r=2$. Since $T^*T={T^2}^*T^2$, by Theorem \ref{t410} and $\|T\|=\|A\|>1$,
$T$ is unitarily similar to
\[
 \begin{pmatrix}
\alpha & b\\
0 & 0
\end{pmatrix},   \mbox{ where }
 |\alpha|=1 \mbox{ and } b\ne 0.
\]
Hence, without loss of generality, we may assume that $T$ has the above form.
Since $\|\widetilde{A}\| = \|T\| = \sqrt{1+|b|^2}$ and $\|\widetilde{A}^*\| \ge \|\widetilde{A}^* e_1\|$,
we deduce that
\[
D = \begin{pmatrix}
0 & \cdots & 0\\
d_1 & \cdots & d_{n-2}
\end{pmatrix},
\]
and thus the first row of $\widetilde{A}$ is $(\alpha \ \ b \ \ 0 \ \ \dots \ \ 0)$. It follows that the first row of $\widetilde{A}^2$ is
$$(\alpha^2 \ \ \alpha b \ \ bd_1 \ \ \dots \ \ bd_{n-2}).$$
Since $\|\widetilde{A}^2\| =\|\widetilde{A}\|=\sqrt{1+|b|^2}$, $|\alpha|=1$ and $b\neq 0$, it forces that $d_j=0$ for all $1\le j\le n-2$.
Thus $D=0$ as asserted.

Finally, if $A$ is nonsingular, then $T$ is also nonsingular. It follows that
$$I_r=(T^*)^{-1}T^*T(T^{-1})= (T^*)^{-1}{T^2}^*T^2(T^{-1})=T^*T,$$
or $T$ is unitary. Thus $\|A\|=\|T\|=1$, contradicting $\|A\|>1$. Hence $A$ is singular, this completes the proof.
\end{proof}

We have the following consequence immediately.

\begin{corollary}\label{cor412}
Let $A$ be an $n$-by-$n$ nonsingular matrix, and let $m$ be the degree of the minimal polynomial of $A$.
The following conditions are equivalent:
\begin{itemize}
\item[\rm (a)] There exists a unit vector $x \in \mathbb{C}^n$ such that $\|A^k x\| = \|A^k\| = \|A\|$ for all $1\le k \le m+1$.
\item[\rm (b)] $A$ is unitarily similar to $[a] \oplus C$ with $\|C\| \le |a| = 1$.

\item[\rm (c)] There exists a unit vector $x \in \mathbb{C}^n$ such that $\|A^k x\| = \|A^k\| = \|A\|$ for all $k \ge 1$.
\end{itemize}
\end{corollary}

\begin{proof}
The implications (b) $\Rightarrow$ (c) $\Rightarrow$ (a) are trivial. Assume (a) holds, $\|A\|=\|A^2\|\le \|A\|^2$ implies that $\|A\|\ge 1$. On the other hand, since $A$ is nonsingular,  Theorem \ref{t44} yields that $\|A\|\le 1$. Therefore, $1=\|A\|=\|A^k\|$ for all $1\le k\le m+1$, and hence (b) follows from \cite[Theorem 2.1]{P}.
This completes the proof.
\end{proof}

For Theorem \ref{t44}, the next proposition shows that the condition
\[
\|A^k x\| = \|A^k\| = \|A\| > 1 \quad \mbox{for } k=1,2,\ldots,m+1,
\]
cannot, in general, be reduced to $k \le m$.

\begin{proposition}\label{prop48}
For $n \ge 2$, let $A \in M_n$ be of the form
\[
A =
\begin{pmatrix}
0 &        &        &        & 1\\
2 & 0      &        &        &  \\
  & 1      & 0      &        &  \\
  &        & \ddots & \ddots &  \\
  &        &        & 1      & 0
\end{pmatrix}.
\]
Then the following statements hold:
\begin{itemize}
\item[\rm (a)] The minimal polynomial of $A$ has degree $n$.

\item[\rm (b)] $\|A^k e_1\| = \|A^k\| = 2$ for $k=1,2,\ldots,n$.

\item[\rm (c)] $A$ is not unitarily similar to a matrix of the form
\[
\begin{pmatrix}
T & D\\
0 & C
\end{pmatrix},
\]
where $T$ satisfies the condition in Theorem \ref{t44} (b).
\end{itemize}
\end{proposition}

\begin{proof}
(a) A direct computation shows that the characteristic polynomial of $A$ is
\[
\det(xI_n-A) = x^n - 2
= \prod_{k=0}^{n-1} \bigl(x - 2^{1/n} e^{i\frac{2k\pi}{n}}\bigr).
\]
Hence $A$ has $n$ distinct eigenvalues
\begin{equation}\label{eq49}
2^{1/n},\; 2^{1/n} e^{i\frac{2\pi}{n}},\; \ldots,\;
2^{1/n} e^{i\frac{2(n-1)\pi}{n}}.
\end{equation}
It follows that the minimal polynomial of $A$ is $x^n - 2$, and thus has degree $n$.

\medskip

(b) A direct computation yields
\[
A^*A = \operatorname{diag}(4,1,\ldots,1), \quad
{A^2}^*A^2 = \operatorname{diag}(4,1,\ldots,1,4),
\]
\[
\ldots, \quad
{A^{(n-1)}}^*A^{n-1} = \operatorname{diag}(4,1,4,\ldots,4), \quad
{A^n}^*A^n = 4I_n.
\]
Thus
 $\|A^k e_1\| = \|A^k\| = 2$ for $k=1,2,\ldots,n$.

\medskip

(c) By \eqref{eq49}, all eigenvalues of $A$ have modulus $2^{1/n}$.
Thus (c) follows from Theorem \ref{t410} (iii).
\end{proof}

Let $T\in M_r$ satisfy $T^*T={T^2}^*T^2$, and let
\[
A=\begin{pmatrix}
T & D\\
0 & C
\end{pmatrix}.
\]
By Theorem \ref{t410} (b), we have
\[
\|T\|=\|T^k\|\le \|A^k\| \ \ \mbox{ for all } k\ge 1.
\]
Note that in Theorem \ref{t44}, if $r=2$ and $\|A^k\|=\|T\|$ for all $k\ge 1$, then $D=0$. A natural question arises: can Theorem \ref{t44} be extended to the case $r\ge 3$? The answer is negative. A counterexample is given below.

\begin{proposition}\label{prop49}
Let
\[
A=\begin{pmatrix}
1&  0&  1&  0& \frac{1}{4}\\
0&  -1&  1&  0& -\frac{1}{4}\\
0&  0&  0&  0& \frac{1}{4}\\
0&  0&  0&  \frac{1}{4}& \frac{1}{4}\\
0&  0&  0&  0& 0
\end{pmatrix},\quad
T=\begin{pmatrix}
1&  0& 1\\
0&  -1& 1\\
0&  0& 0
\end{pmatrix},
\]
and let $x=\frac{1}{\sqrt{6}}[1 \ -1 \ 2 \ 0 \ 0]^t\in\IC^6$. Then the following hold:
\begin{itemize}
  \item[\rm (a)] $T^*T={T^2}^*T^2$.
  \item[\rm (b)] $\|A^k x\|=\|A^k\|=\|A\|=\|T\|=\sqrt{3}$ for all $k\ge 1$.
  \item[\rm (c)] $A$ is unitarily irreducible.
\end{itemize}
\end{proposition}

\begin{proof}
(a) A direct computation shows that
\[
T^*T=
\begin{pmatrix}
1&  0& 1\\
0&  1& -1\\
1&  -1& 2
\end{pmatrix}
= {T^2}^*T^2.
\]

(b) For $k\ge 1$, one computes
\[
A^{2k}=
\begin{pmatrix}
1&  0&  1&  0& \frac{1}{2}\\
0&  1&  -1&  0& \frac{1}{2}\\
0&  0&  0&  0& 0\\
0&  0&  0&  \frac{1}{4^{2k}}& \frac{1}{4^{2k}}\\
0&  0&  0&  0& 0
\end{pmatrix}
\quad \mbox{and} \quad
A^{2k+1}=
\begin{pmatrix}
1&  0&  1&  0& \frac{1}{2}\\
0&  -1&  1&  0& -\frac{1}{2}\\
0&  0&  0&  0& 0\\
0&  0&  0&  \frac{1}{4^{2k+1}}& \frac{1}{4^{2k+1}}\\
0&  0&  0&  0& 0
\end{pmatrix}.
\]
Moreover,
\[
A^*A=
\begin{pmatrix}
1&  0&  1&  0& \frac{1}{4}\\
0&  1&  -1&  0& \frac{1}{4}\\
1&  -1&  2&  0& 0\\
0&  0&  0&  \frac{1}{16}& \frac{1}{16}\\
\frac{1}{4}&  \frac{1}{4}&  0&  \frac{1}{16}& \frac{1}{4}
\end{pmatrix}.
\]
For $k\ge 1$, one verifies that
\[
{A^k}^*A^k x = 3x
\quad \mbox{and} \quad
T^*T x' = 3x',
\]
where $x'=\frac{1}{\sqrt{6}}[1 \ -1 \ 2]^t\in\IC^3$. Hence $3$ is an eigenvalue of both ${A^k}^*A^k$ and $T^*T$.

Furthermore,
\[
\operatorname{tr}(T^*T)=4<6,
\]
and
\[
\operatorname{tr}({A^k}^*A^k)\le \operatorname{tr}({A^2}^*A^2)
=4+\frac{1}{2}+\frac{2}{4^4}<6.
\]
It follows that $3$ is the largest eigenvalue of both ${A^k}^*A^k$ and $T^*T$. Therefore, we obtain
\[
\|T\|^2=\|T^*T\|=3
\]
and
\[
\|A^k x\|^2
=\langle {A^k}^*A^k x,x\rangle
=3\|x\|^2
=3
=\|{A^k}^*A^k\|
=\|A^k\|^2
\]
for all $k\ge 1$.

(c) It suffices to show that the only projections commuting with $A$ are $0_5$ and $I_5$. Let $P$ be a $5\times 5$ Hermitian projection such that $PA=AP$. Write
\[
P=(p_{ij})_{i,j=1}^5.
\]
Comparing the $(i,j)$-entries of $AP$ and $PA$, we obtain successively
\[
p_{15}=p_{25}=p_{35}=p_{45}=0,
\]
\[
p_{14}=p_{24}=p_{34}=0,\quad p_{55}=p_{44},
\]
\[
p_{13}=p_{23}=0,\quad p_{33}=p_{55},
\]
\[
p_{12}=0,\quad p_{22}=p_{55},
\]
\[
p_{11}=p_{55}.
\]
Hence $P=p_{55}I_5$. Since $P=P^2=P^*$, it follows that $P=0_5$ or $P=I_5$. Therefore, $A$ is unitarily irreducible.
\end{proof}

Finally, we construct a $3$-by-$3$ unitarily irreducible matrix $T$ satisfying
$T^*T = {T^2}^*T^2$, which provides a counterexample to Conjecture \ref{c11}.

\begin{proposition}\label{prop413}
Let
\[
T =
\begin{pmatrix}
e^{i\sqrt{2}\pi} & 0 & \frac{1}{\sqrt{2}}\\
0 & e^{-i\sqrt{2}\pi} & \frac{1}{\sqrt{2}}\\
0 & 0 & 0
\end{pmatrix}.
\]
Then the following statements hold:
\begin{itemize}
\item[\rm (a)] There exists a unit vector $x$ such that
\[
\|T^k x\| = \|T^k\| = \sqrt{2} \quad \text{for all } k \ge 1.
\]

\item[\rm (b)] $T^k$ is unitarily irreducible for all $k \ge 1$.

\item[\rm (c)] For all $k \ge 1$, $T^k$ is not a modulus-one multiple of an idempotent.
\end{itemize}
\end{proposition}

\begin{proof}
(a) In Theorem \ref{t410}, take
\[
U = \begin{pmatrix}
e^{i\sqrt{2}\pi} & 0 \\
0 & e^{-i\sqrt{2}\pi}
\end{pmatrix}
\quad \mbox{and} \quad
B = \begin{pmatrix} \frac{1}{\sqrt{2}} \\[4pt] \frac{1}{\sqrt{2}} \end{pmatrix}.
\]
Since $T= \begin{pmatrix}
U & B \\
0 & 0
\end{pmatrix}$ and $U$ is unitary, the desired conclusion follows from Theorem \ref{t410}.

\medskip

(b) For each $k \ge 1$, the eigenvalues of $T^k$ are
$
e^{ik\sqrt{2}\pi},  e^{-ik\sqrt{2}\pi}$, and $0$.
A direct computation shows that the corresponding eigenvectors can be chosen as
$
e_1,e_2$, and
\[
v = \begin{pmatrix} 1 \\ e^{i2\sqrt{2}\pi} \\ -\sqrt{2}e^{i\sqrt{2}\pi} \end{pmatrix}.
\]
In particular, we have
\[
\langle v, e_1 \rangle \neq 0 \quad \text{and} \quad \langle v, e_2 \rangle \neq 0.
\]
Hence, there is no nontrivial reducing subspace spanned by a subset of eigenvectors, and therefore $T^k$ is unitarily irreducible for all $k \ge 1$.

\medskip

(c) This follows immediately from (b) and \cite[Lemma 3.1]{GWW}.
\end{proof}

Here, \cite[Lemma 3.1]{GWW} says that if $X\in M_n$ is an idempotent and $n\ge 3$, then $X$ is unitarily reducible.

\medskip\noindent
{\bf \large Acknowledgment}

The researches of  Gau and  Wang were partially supported by National Science and Technology Council of Taiwan under grant NSTC 114-2115-M-008-005 and NSTC 114-2115-M-A49-005, respectively.
Li is an affiliate member of the Institute for Quantum Computing, University
of Waterloo. His research was supported by the Simons Foundation Grant 851334.

\medskip\noindent
(Gau) Department of Mathematics, National Central University, Chungli 32001, Taiwan. hlgau@math.ncu.edu.tw

\medskip\noindent
(Hong) Department of Mathematics, National Tsing Hua University, Hsinchu 300044, Taiwan. a3229868@gmail.com

\medskip\noindent
(Li) Department of Mathematics, College of William \& Mary, Williamsburg, VA 23187, USA. ckli@math.wm.edu

\medskip\noindent
(Wang) Department of Applied Mathematics, National Yang Ming Chiao Tung University, Hsinchu 300093, Taiwan.
kzwang@math.nctu.edu.tw

\end{document}